\documentclass[11pt]{article}
\usepackage[letterpaper, margin=1in]{geometry}
\usepackage[T1]{fontenc}

\usepackage{amsmath, amssymb, amsthm, authblk, bbm, mathtools, subcaption}
\usepackage[sort&compress,numbers]{natbib}
\usepackage[linktocpage=true]{hyperref}
\hypersetup{colorlinks=true, linkcolor=blue}

\allowdisplaybreaks
\numberwithin{equation}{section}

\usepackage{titlesec}
\usepackage{algorithm,algpseudocode}
\usepackage{enumerate}
\usepackage{float}
\usepackage[toc,page]{appendix}
\usepackage{bm}
\mathtoolsset{showonlyrefs}

\numberwithin{equation}{section}
\newtheorem{theorem}{Theorem}[section]
\newtheorem{corollary}[theorem]{Corollary}
\newtheorem{definition}[theorem]{Definition}
\newtheorem{proposition}[theorem]{Proposition}
\newtheorem{lemma}[theorem]{Lemma}
\newtheorem{remark}[theorem]{Remark}

\newcommand{\keywords}[1]{\textbf{\textit{Keywords ---}} #1}

\newcommand{\E}{\mathbb{E}}
\newcommand{\R}{\mathbb{R}}

\newcommand{\F}{\mathcal{F}}

\DeclareMathOperator*{\argmin}{\arg\!\min}

\newcommand{\norm}[1]{\left\|#1\right\|_2}
\newcommand{\ai}{{\textsc{lai}}}
\newcommand{\robd}{{\textsc{robd}}}
\newcommand{\FI}{{\textsc{fi}}}
\newcommand{\LAItxt}{{\textsc{lai}$(\gamma)$}}
\newcommand{\LAImath}{{\textsc{lai}(\gamma)}}
\newcommand{\nofi}{{\textsc{lai}(1)}}
\newcommand{\ftm}{{\textsc{ftm}}}
\newcommand{\alg}{{\textsc{alg}}}
\newcommand{\regret}{{\text{Regret}}}

\begin{document}

\title{Best of Both Worlds Guarantees for \\Smoothed Online Quadratic Optimization\footnote{This work was partially supported by the NSF grants CIF-2113027 and CPS-2240982. }}
\date{}

\author[1]{Neelkamal Bhuyan\thanks{Email: \href{mailto:nbhuyan3@gatech.edu}{nbhuyan3@gatech.edu}}\textsuperscript{,}}
\author[1]{Debankur Mukherjee}
\author[2]{Adam Wierman}

\affil[1]{Georgia Institute of Technology}
\affil[2]{California Institute of Technology}

\maketitle

\begin{abstract}
We study the smoothed online quadratic optimization (SOQO) problem where, at each round $t$, a player plays an action $x_t$ in response to a quadratic hitting cost and an additional squared $\ell_2$-norm cost for switching actions. This problem class has strong connections to a wide range of application domains including smart grid management, adaptive control, and data center management, where switching-efficient algorithms are highly sought after. We study the SOQO problem in both adversarial and stochastic settings, and in this process, perform the first stochastic analysis of this class of problems. We provide the online optimal algorithm when the minimizers of the hitting cost function evolve as a general stochastic process, which, for the case of martingale process, takes the form of a \emph{distribution-agnostic dynamic interpolation algorithm} (\ai{}). Next, we present the stochastic-adversarial trade-off by proving an $\Omega(T)$ expected regret for the adversarial optimal algorithm in the literature (\robd{}) with respect to \ai{} and, a sub-optimal competitive ratio for \ai{} in the adversarial setting. 
Finally, we present a best-of-both-worlds algorithm that obtains a robust adversarial performance while simultaneously achieving a near-optimal stochastic performance.
\end{abstract}

\section{Introduction}
We study a class of smoothed online quadratic optimization (SOQO) problems in which a player has to make an online decision $x_t \in \R^d$, in response to a quadratic hitting cost  $f_t(x) = \frac{1}{2}(x-v_t)^T A (x-v_t)$ and $\ell_2$-norm switching cost, i.e., $c(x_t, x_{t-1})=\frac{1}{2}\|x_t - x_{t-1}\|_2^2$. We consider the problem over a finite horizon $T$, where the sequence of minimizers $\{v_t\}_{t=1}^T$ can be \textit{adversarial} or \textit{stochastic} and is revealed in an online fashion to the player.

In the last decade, the online optimization problem with quadratic costs has received significant attention due to its applications in several domains, including smart grid cost management \cite{KimGiannakis17,WangHuang14,NarayanaswamyGarg12,BadieiWierman15,Mei2011Grid,Wang1993ramp,SantosLee08}, adaptive control~\cite{Shin2023,Li2022,Li2021,GoelWierman19, Csaba11,Tang2021,Lale2020,Cohen2019,Dean2018,GoelChenWierman17,LiGuannan18},
data center management \cite{LinWiermanAndrew11,LinWierman12,LuAndrew13}, electrical vehicle charging \cite{KimMatthews15,GanTopcu13}, video transmission \cite{JosephVeciana12}, power systems \cite{NarayanaswamyGarg12,LuTuChau13,KimGiannakis17}, and chip thermal management \cite{ZaniniAtienza09,ZaniniAtienza10}. 

However, since its introduction, the smoothed online optimization problem has been studied primarily from an adversarial point of view \cite{AndrewWierman13,BansalGupta15,ChenWierman15,Antoniadis16,ChenWierman16,ChenWierman18, Antoniadis18,GoelLinWierman19, ZhangYang21,ZhangYang22, RuttenMukherjee23, ChristiansonWierman22,ChristiansonWierman23}. 
Attempts to incorporate noise in the related domain of LQR control consider very specific distributions and dynamics of the system \cite{ChenWierman15,ChenWierman16}, most of which cannot be translated to SOQO. Consequently, the existing algorithms in the online optimization community are overly pessimistic for situations that may have more structure and potentially a stochastic behavior. 

In related domains, such as scheduling, stochastic modeling has been used to capture known data patterns, and algorithms can then be designed to perform well for these known models \cite{Koutsopoulos2011,Urgaonkar2011,Jennings1996,Gandhi2010,JosephVeciana12}.  In contrast, adversarially designed algorithms must safeguard against arbitrary inputs and thus cannot exploit the inherent structures within the environment. 
This trade-off between average-case and worst-case performance is long-standing in the algorithms literature~\cite{Sriskandarajah89,Matsuo90,Kuhn03,Robey22}. 

In the case of SOQO, the design of algorithms for the case of stochastic hitting cost functions is unexplored, with this work being the first to address it. Since much of the literature in smoothed online optimization focuses on worst-case performance, the algorithms and analysis methods are not geared to exploit the structure which is now present in a stochastic environment. This motivates us to address the stochastic SOQO problem first in the quadratic setting, answering the following question:

\begin{quote}
    \textit{``Is there a simple characterization of a near-optimal policy for stochastic SOQO?  How do existing policies, designed primarily for the adversarial setting, compare to this (near-)optimal policy?"} 
\end{quote}
Tailoring to stochastic setting may lead to lack robustness against adversarial input. This represents the other side of the long-standing trade-off between optimizing average-case and worst-case performance, and motivates us to answer the second foundational question for SOQO:

\begin{quote}
    \textit{``Is there an algorithm for SOQO that achieves near-optimal performance simultaneously in stochastic and adversarial settings?"} 
\end{quote}

Such \emph{best-of-both-worlds} algorithms are akin to the seminal work by Bubeck and Slivkins~\cite{BS12}, which answers the question affirmatively in the multi-arm bandit setting.  Best-of-both-worlds algorithms are highly sought after in the algorithms literature; however, they are also rare, and there are key characteristics distinguishing the SOQO framework from the multi-arm bandits' framework of~\cite{BS12}, e.g., the \textit{unbounded continuous action space} and the \textit{dynamic optimal performance benchmark} (as opposed to a focus on static regret), that make it questionable whether such a best-of-both-worlds algorithm exists for SOQO. \\

\noindent
\textbf{Contributions.} This paper provides answers to both of the questions highlighted above.  Specifically, the contributions of this paper are four-fold:\vspace{0.2cm}

\noindent
(a) \emph{Characterization of a stochastic online-optimal algorithm}:
We start by considering the SOQO problem in the setting when the sequence of minimizers $\{v_t\}_t$ (of the hitting costs $\{f_t\}_t$) evolves as a stochastic process. We analyze the optimization problem through the lens of dynamic programming and obtain a \textit{closed form solution} for the online optimal algorithm (Theorem \ref{thm:gen_stochastic}). 
In particular, when $\{v_t\}_t$ forms a martingale, our optimal solution translates to a \textit{distribution-agnostic} algorithm, which we call the \emph{Lazy Adaptive Interpolation} (\ai{}, Algorithm \ref{alg:LAI_martingale}).
We further characterize its total cost, observe it to be $\mathcal{O}(T)$ (which is optimal), and to be insensitive to the distribution beyond the variance (see Theorem \ref{thm:LAI_total_cost}). 
Notably, this algorithm is dynamic in nature and, as expected, has a much better performance than any static hindsight-optimal decision (see Theorem \ref{thm:hindsight_static_opt}), which exhibits a $\Theta \left(T^2\right)$ cost. 
\vspace{0.2cm}

\noindent
(b) \emph{Stochastic suboptimality of the adversarial optimal}:
When adversarial inputs are considered, the performance is measured in terms of the competitive ratio, the worst-case ratio of the cost of an algorithm to the cost of the hindsight optimal sequence of decisions. 
In~\cite{GoelLinWierman19}, the authors propose the \robd{} algorithm and show that it has an optimal competitive ratio among all online algorithms.
When the inputs are martingale, however, we show that the \robd{} algorithm admits a linear regret (in the time horizon $T$) with respect to the stochastic optimal cost, i.e., the cost of the aforementioned \ai{} algorithm (see Theorem \ref{thm:R-OBD_lin_regret}). 

\noindent
(c) \emph{Adversarial analysis of stochastic optimal}:
To understand the robustness of the \ai{} algorithm against adversarial input, we perform its competitive analysis and find it potentially suboptimal (see Theorem \ref{thm:AI_cr}).
En route, we also establish competitive ratio bounds for a much wider class of memoryless algorithms, which might be of independent interest. The novelty of our analysis framework lies in its applicability to algorithms with decision rules that change over time (see Theorem \ref{thm:adv_analysis_framework}). \vspace{0.2cm}

\noindent
(d) \emph{Best-of-both-worlds algorithm design}:
Our final contribution focuses on the design of a best-of-both-worlds algorithm for SOQO. 
We present a novel hyperparametric algorithm called \LAItxt{} (Algorithm~\ref{alg:NOFI_AI}) with $\gamma\in [0,1]$, which delivers near-optimal performance (see Corollary \ref{corr:LAI_best_of_both}) in both stochastic and adversarial environments without leveraging prior knowledge about the nature of the environment.
Specifically, the hyperparameter $\gamma\in [0,1]$ determines the algorithm's position in the trade-off between adversarial and martingale settings. 
For any choice of $\gamma\in [0,1]$, the algorithm guarantees 
{(i)} a finite regret (that depends on $\gamma$) with respect to the online-optimal \ai{} algorithm, and
{(ii)} a finite competitive ratio (that depends on $\gamma$) very close to that of adversarial-optimal \robd{} algorithm (see Theorem \ref{thm:NOFI_AI_perf}).\vspace{0.3cm}

\textit{Analytical Contributions:} Analytical tools in the literature of smoothed online optimization focus solely on optimizing over the current cost function, with no regard to the future evolution of the minimizer. Further, works like \cite{Lin2021} argue that finding the optimal action sequence and cost-to-go functions is near-impossible for general cost functions. In light of these, and to obtain insight into stochastic smoothed online quadratic optimization, we make a number of contributions on the methodological front.

In particular, at the core of our analysis, we characterize the online optimal algorithm in the stochastic setting by leveraging the dynamic programming (DP) machinery. Although the DP framework offers a powerful set of tools, a fundamental difficulty lies in understanding the structure of the value function. This is compounded by the fact that our setting involves \textit{finite horizon and unbounded continuous state space and action space}. The horizon being finite makes the value function's structure \textit{time-variant}, decoding which was one of the major hurdles. The infinite state/action space ruled out tabular methods and demanded a \textit{closed form} for the time-variant value function. We overcame these by inductively identifying structure in the value function (see Proposition \ref{prop:AI_val_fn} and Theorem \ref{thm:value_fn_derivative}), enabling us to express the optimal algorithm in a simple closed form. Beyond our stochastic analysis, in the adversarial setting, the non-stationary nature of our algorithms necessitated an extension of the existing adversarial analysis methods to encompass a broader range of algorithms, allowing tighter analysis than that provided by existing methods. Lastly, the design of \LAItxt{} required careful weaving of properties of both stochastic optimal and adversarial optimal, leading to a smooth transition between the two worlds.

\section{Model and Preliminaries}
\label{sec:problem_setup}
We study smoothed online quadratic optimization (SOQO) and consider an action space $\R^d$, $d \geq 1$, over a finite time horizon of $T$ rounds. 
The hitting cost for round $t$ is given by $f_t(x) = \frac{1}{2}(x-v_t)^T A (x-v_t)$, where $A$ represents a known, positive definite $d \times d$ matrix, denoted as $A \succ 0$, and $v_t$'s are revealed in an online fashion to the decision maker. 
Additionally, there is a switching cost of $\frac{1}{2}\|x_t - x_{t-1}\|_2^2$ as the player transitions between actions.
Therefore, the player aims to minimize $\sum_{t=1}^T \big(f_t(x_t) + \frac{1}{2}\|x_t - x_{t-1}\|_2^2\big)$ in an online fashion.
Our work explores SOQO in both adversarial and stochastic environments.

\subsection{Adversarial SOQO}
Prior work on smoothed online optimization has considered an adversarial setting in which algorithms seek to ensure a performance guarantee against any sequence of minimizers $\{v_t\}_{t=1}^T$. 
Performance in this context is assessed most commonly using the \textbf{competitive ratio}, defined for an online algorithm \alg{} as
\emph{the worst-case ratio of the total cost of \alg{} and that of the offline optimal sequence of decisions}. 
For any finite time horizon $\{1, \ldots, T\}$, we denote the cost of \alg{} and the hindsight optimal solution by 
$\text{Cost}_{\text{\alg{}}}[1,T]$ and $\text{Cost}_{\text{OPT}}[1, T]$, respectively.
Within the class of hitting and switching costs considered in this work, recent work in \cite{GoelLinWierman19} has identified an algorithm, \robd{}, which maintains the optimal competitive ratio possible by any online algorithm. \robd{} is defined as follows. 
\begin{definition}\label{def:robd} \textit{
The action of \robd{} in round $t$ is  
$
    x_t = \argmin_x f_t(x) + \frac{\mu_1}{2}\norm{x - x_{t-1}}^2 + \frac{\mu_2}{2}\norm{x - v_t}^2,
$ where $\mu_1 \in [0,1]$ and $\mu_2 > 0$ are hyperparameters. 
Optimizing over $\mu_1$ and $\mu_2$, the \robd{} algorithm in our setting is
$    x_t = C_{\robd{}} x_{t-1} + (I-C_{\robd{}})v_t$
where 
$    C_{\robd{}} = \left(A + I + \mu_2^*\left(\lambda_{\min}^A\right) I \right)^{-1},$
with $\mu_2^*\left(\lambda_{\min}^A\right) = \frac{\lambda^A_{\min}}{2}\left( \sqrt{1 + \frac{4}{\lambda^A_{\min}}} - 1 \right)$ and $\lambda_{\min}^A$ denoting the smallest eigenvalue of $A$. }
\end{definition}

\noindent The competitive ratio of \robd{}, as defined above, is given by $1 + \frac{1}{2}\left(\sqrt{1 + \frac{4}{\lambda_{\min}^A}} - 1\right)$ \cite{GoelLinWierman19}.
A notable property of this competitive ratio used in this paper is its $1/\sqrt{\lambda_{\min}^A}$ growth as $\lambda_{\min}^A\to 0$, which has been shown to be optimal in~\cite{GoelLinWierman19}.

\begin{remark}\label{remark:low_lambda}
\textit{The limiting behavior of the competitive ratio under $\lambda_{\min}^A\to 0$ (or $\lambda_{i}^A\to 0$ $\forall$ $i$) is the focus of smoothed online optimization literature~\cite{BansalGupta15,Antoniadis16,Antoniadis18,ChenWierman18,GoelLinWierman19,GoelWierman19,ZhangYang21,ChristiansonWierman22, ChristiansonWierman23,RuttenMukherjee23}. Here, hitting costs are very low compared to switching costs which renders follow-the-minimizer algorithm as sub-optimal, demanding search for non-trivial algorithms.}
\end{remark}

\subsection{Stochastic SOQO}
We introduce and, for the first time, study a stochastic version of SOQO. The stochasticity in the problem arises from the sequence of minimizers $\{v_t\}_{t=1}^T$ being a stochastic process, with $\mathcal{F}_t = \sigma\left(\{v_s\}_{s=1}^t \right)$, $t= 1, \ldots, T$, being the natural filtration generated by $\{v_s\}_{s=1}^t$. In this work, we employ the dynamic programming (DP) machinery in the stochastic context, introduced by Bellman \cite{Bellman66}, as a method to find the optimal \textit{online} action sequence with respect to the expected total cost over the horizon. 
Here, the optimal online action at round $t$, $x^{DP}_t$, is given as the minimizer of
\begin{equation}\label{eqn:value_fn_defn}
     \min_{x \in \R^d} \bigg\{f_t(x) + c\left(x,x^{DP}_{t-1} \right) + \E \left[V_{t+1}(x,v_{t+1})|\mathcal{F}_t \right] \bigg\},
\end{equation}
where $V_{t}\left(x^{DP}_{t-1},v_{t}\right)$, called the \textit{value function}, is the minimum online expected current and future cost (or cost-to-go) at round $t$.
The algorithm derived from solving the aforementioned problem represents the online optimal.
It is worthwhile to note that the optimal online algorithm is usually characterized implicitly using the above-mentioned DP approach. However, in most cases, it is non-trivial and sometimes impossible to formulate a simple and explicit algorithm that is online-optimal \cite{Lin2021}. We achieve this is Section~\ref{ssec:opt-stoch}.

We put special focus on the case where the sequence of minimizers $\{v_t\}_{t=1}^T$ forms a martingale in $\R^d$, i.e, 
\begin{equation}\label{eqn:minimizer_martingale}
    \E[v_s | \mathcal{F}_t] = v_t \text{ } \forall \text{ } s\geq t.
\end{equation}
In this case, we will design an algorithm that can simultaneously achieve near-optimal performance both in the stochastic (in terms of regret with respect to the online optimal algorithm) and in the adversarial case -- thus, furnishing a best-of-both worlds performance.
The increments to the minimizers, that is, $v_t - v_{t-1}$, are considered to have a finite covariance, denoted as $\Sigma_t$. 
With a slight abuse of notation, we define the `variance' of $v_t - v_{t-1}$ as $\sigma_t^2 := \E \|v_t - v_{t-1}\|_2^2 = tr(\Sigma_t)$.

For the stochastic setting, we use the \textbf{dynamic online optimal algorithm} as the performance benchmark instead of the \emph{static} hindsight optimal commonly used in many online optimization works. 
This is because the latter exhibits a poor performance in the current setting, with a total cost that is an order of magnitude higher than that of the online optimal (Theorem \ref{thm:hindsight_static_opt}). Such a benchmark is not uncommon in the literature, as evidenced by~\cite{VeraBanerjee21}. 
We thus, consider the \textbf{(dynamic) regret as the performance metric} in the stochastic case: For any online algorithm \alg{} define
\begin{align}\label{eqn:regret_defn}
    \regret_{\text{\alg{}}}[1,T] := \E[\text{Cost}_{\text{\alg{}}}[1,T]] - \E[\text{Cost}^*[1,T]],
\end{align}
where $\E[\text{Cost}^*[1,T]]$ is the total cost of the online optimal algorithm in the stochastic setting.

\begin{remark} \textit{
The stochastic and the adversarial settings have different metrics due to the different behaviors of the environments. Although an additive gap is preferred over competitive ratio, the worst-case additive gap to $OPT$ in the adversarial setting is unbounded for finite $T$, unlike the stochastic environment. This concept of different metrics for separate assumptions already exists in the literature \cite{ZhangYang21, ChenWierman18,GoelLinWierman19}.}
\end{remark}

\section{Algorithms and Main Results}
We now present our main contributions, which include a characterization of an online optimal stochastic SOQO algorithm and an analysis of its performance in both stochastic and adversarial settings, an analysis of an optimal adversarial SOQO algorithm (\robd{}) in the martingale setting, and the design and analysis of a best-of-both-worlds algorithm that is near-optimal in both martingale and adversarial settings.

\subsection{Optimal Stochastic SOQO}
\label{ssec:opt-stoch}
The optimal online algorithm in the stochastic setting \eqref{eqn:minimizer_martingale} can be characterized using the dynamic programming (DP) machinery \eqref{eqn:value_fn_defn} discussed in the preliminaries. Our primary contribution lies in simplifying the formulation of the optimal action. The simple form we obtain is a departure from the complexity typically associated with the underlying DP methodology.

The \textit{online optimal} action, as described by Algorithm \ref{alg:LAI_martingale}, is an interpolation between $v_t$ and the previous action. The crux of the algorithm lies in the recursion followed by the matrix sequence $\{C_t\}_t$, that is,
$
    C_t^{-1} = 2I + A - C_{t+1},
$
obtained by solving the DP problem in the SOQO context. The significance of this recursion lies in its connection to (near-)optimal stochastic performance for the SOQO setting we consider. Our first result states the optimality of \ai\ and provides a closed form characterization of its optimal cost.

\begin{algorithm} [t]
\caption{Lazy Adaptive Interpolation}\label{alg:LAI_martingale}
\flushleft \textbf{Input:} $A \succ 0$, $b \in \R^d$, $T$
\flushleft \textbf{Initialize:} $\{C_t\}_t$ such that $C_t^{-1} = 2I + A - C_{t+1}$ $\forall$ $t \in \{1,\ldots, T-1\}$ and $C_T = (I+A)^{-1}$
\begin{algorithmic}
\Procedure{\ai{}}{}
\For{$t = 1,2,\ldots,T$}
\State $x^{\ai{}}_t \gets C_t  x^{\ai{}}_{t-1} + (I-C_t)  v_t$
\EndFor
\EndProcedure
\end{algorithmic}
\end{algorithm}

\begin{theorem}\label{thm:LAI_total_cost}
Lazy Adaptive Interpolation \ai{} is \textbf{online optimal} in the stochastic setting \eqref{eqn:minimizer_martingale} with total cost
\begin{equation*}
\begin{split}
    \E[\text{Cost}_{\text{\ai{}}}[1,T]] &= \sum_{t=1}^{T} \frac{1}{2} \E\left[ (v_t - v_{t-1})^T (I-C_{t}) (v_{t} - v_{t-1}) \right] \\
    &\leq \sum_{t=1}^{T} \frac{1}{2} \E\left[ (v_t - v_{t-1})^T (I-C_{L}) (v_{t} - v_{t-1}) \right]
    \end{split}
\end{equation*}
where $C_L = \frac{A + 2I - \sqrt{A^2 + 4A}}{2}$ and $v_0 = x_0$.
\end{theorem}
Notice that the optimal cost in the martingale setting, described above, grows with the horizon as $\mathcal{O}(T)$.

The fact that $\{C_t\}_t$ depends solely on $A$ and the horizon $T$ unveils an unexpected aspect of the optimal online action: its insensitivity to the specifics of the stochastic environment. The proof of \ref{thm:LAI_total_cost} is presented in detail in Appendix \ref{proof:LAI_perf}.

For the martingale setting, we analyze the complete value function to extract its closed form expression. This helps us characterize both the optimal algorithm and its expected cost. Looking deeper into the \textit{derivative} of the value function, we get the following online optimal algorithm:
\begin{theorem}\label{thm:gen_stochastic}
For the general case where $\{v_t\}_{t=1}^T$ is any stochastic process, the optimal action at round $t$ is given by
\begin{align*}
\begin{split}
    x_t =& \text{ } C_t x_{t-1} + (I-C_t)v_t
    + \sum_{s=t+1}^T \left(\prod_{q=t}^{s-1} C_q \right) (I-C_s)\E[v_s - v_{s-1}| \mathcal{F}_t]
\end{split}
\end{align*}
where $\{C_t\}_{t=1}^T$ are the coefficients defined in \ai{}.
\end{theorem}
The proof is presented in Appendix \ref{sec:gen_stochastic}. The most important aspect is the derivative of the value function, which takes a simple yet useful form, stated in Theorem \ref{thm:value_fn_derivative}, which result holds for general convex hitting costs and switching costs.

The optimality claims in Theorems \ref{thm:gen_stochastic} and \ref{thm:LAI_total_cost} are agnostic to distributions (beyond the martingale structure for the latter). This implies that the optimality is robust to two extreme stochastic scenarios: (i) \textbf{distribution shifts} and, (ii) \textbf{heavy tails} (subject to finite variance).

\begin{remark}\label{LAI_uniqueness}\textit{
Note that the LAI algorithm is the \textbf{unique} online-optimal algorithm in this setting. This is because any online-optimal action has to satisfy the value function expression at every round \cite{bellman1954theory}, which has a unique solution due to its quadratic structure.}
\end{remark}

Within the online optimization literature, it is common practice to employ the optimal static action in hindsight as a benchmark to evaluate online algorithms. However, the following result illustrates that such a benchmark is not appropriate within our framework. 
This highlights the necessity to design an algorithm that adapts dynamically in time, in order to obtain good performance guarantees. 

\begin{theorem}\label{thm:hindsight_static_opt} 
Consider the stochastic setting \eqref{eqn:minimizer_martingale} and $f_t(x) = \frac{1}{2}\norm{x - v_t}^2$ with the class of static algorithms where the action is $\underline{x} \in \R^d$ for all rounds. When the increments to the minimizers, that is, $\left\{v_t - v_{t-1} \right\}_t$, have the same variance $\sigma^2$, the cost of such an algorithm, denoted by $\text{Cost}_{\underline{x}}[1,T]$, satisfies
\begin{align}
\E\left[\inf_{\underline{x}}\text{Cost}_{\underline{x}}[1,T] \right] = \Theta \left( T^2 \right)
\end{align}
\end{theorem}
Theorem~\ref{thm:hindsight_static_opt} is proved in Appendix~\ref{proof:hindsight_static_opt}.
This result has two significant implications. \textbf{First}, in contrast with Theorem~\ref{thm:LAI_total_cost}, it highlights the poor performance of the entire class of static algorithms in the stochastic SOQO framework. 
\textbf{Second}, it emphasizes the need for a dynamic benchmark when evaluating online algorithms; thereby leading us to use \ai\ as the benchmark in this work.

\subsection{Stochastic Analysis of ROBD and Fixed Interpolation Algorithms}
Having characterized the stochastic optimal algorithm, our attention shifts to analyzing the performance of the adversarial optimal algorithm, \robd{}, in the \textit{martingale} setting. As previously highlighted, a common concern is that the performance of adversarial algorithms may potentially be sub-par in stochastic settings. The following result demonstrates that \robd{} indeed has markedly poorer performance than the stochastic optimal algorithm \ai{} in the stochastic setting.

\begin{theorem}\label{thm:R-OBD_lin_regret}
Consider the stochastic setting \eqref{eqn:minimizer_martingale}. Additionally, assume that the increments $(v_t - v_{t-1})$ have the same covariance matrix $\Sigma$. 
Then for any $A \neq \lambda I$ (which occurs for $d\geq 2$), the regret of \robd{} is lower bounded as
\begin{align*}
    \text{Regret}_{\robd{}}[1,T]=
    \begin{cases}
        \Omega(T) & \text{ if } A \text{ and } \Sigma \text{ do NOT have same eigenvectors}\\
        \Omega(T) & \text{ otherwise if } (A - \lambda_{\min}^A I)\Sigma \neq 0_{d\times d}
    \end{cases}
\end{align*}
\end{theorem}
The most important takeaway here is the large class of covariance matrices for which the above negative result holds, revealing that \robd{} has linear regret even in very simple stochastic settings. To induce such regret for \robd{}, all that is needed is positive variance along one of the eigenvectors of $A$ that does not correspond to the minimum eigenvalue.
The primary factor that causes \robd{} to have linear regret is the wide gap between the matrix $C_{\robd{}}$ and the matrix sequence of \ai{}, that is, $\{C_t\}_{t=1}^T$. In fact, this problem persists in a large class of algorithms, of which \robd{} is a part.

\begin{definition}
Consider a symmetric matrix $C$ with the same set of eigenvectors as $A$ such that $C \prec I$. The class of Fixed Interpolation (\FI{}) algorithms makes decisions at time $t$ according to
$x_t = C x_{t-1} + (I-C) v_t \text{, } \forall \text{ } t.$

\end{definition}

\noindent It is straightforward to see that \robd{} is an \FI{} algorithm, owing to $C_{\robd{}} \prec I$ and having the same set of eigenvectors as $A$. Although \FI{} algorithms feature simplicity of implementation, the following result demonstrates their poor stochastic performance and, consequently, the linear regret of \robd{}.

\begin{theorem}\label{thm:FI_lin_regret_short}
Consider the stochastic setting \eqref{eqn:minimizer_martingale}. Additionally, assume the increments $(v_t - v_{t-1})$ have the same covariance matrix $\Sigma$. Then for any $C \neq C_L = \frac{A + 2I - \sqrt{A^2 + 4A}}{2}$,
\begin{align*}
    Regret_{\FI{}}[1,T]=
    \begin{cases}
        \Omega(T) & \text{ if } A \text{ and } \Sigma \text{ do NOT have same eigenvectors}\\
        \Omega(T) & \text{ otherwise if } (C - C_L)\Sigma \neq 0_{d\times d}
    \end{cases}
\end{align*}
\end{theorem}

Here, the importance of the recursion followed by $\{C_t\}_{t=1}^T$, in the \ai{} algorithm, is highlighted in the context of fixed interpolation algorithms. Unless the matrix $C$ is $C_L$, one can show that a ``non-vanishing gap" exists between $C$ and $\{C_t\}_{t=1}^T$, which becomes the driving factor for the linear regret. We present the proof of the above theorem along with the exact expression of the lower bound, which quantifies the factors responsible for the linear regret, in Theorem \ref{thm:FI_lin_regret} in Section \ref{proof:FI_lin_regret_short}.  Note that the matrix $C_L$, in fact, has a vanishing gap to $\{C_t\}_{t=1}^T$ and we revisit this important property in upcoming subsections.

\subsection{Adversarial Analysis of Lazy Adaptive Interpolation Algorithms}\label{subsec:adversarial}
We now present the other half of the contrast between the stochastic optimal and adversarially optimal algorithms.  Specifically, in this section, we bound the performance of the stochastic optimal policy \ai{} in the adversarial setting.

\begin{theorem}\label{thm:AI_cr}
In the adversarial setting, the competitive ratio of the \ai{} algorithm for hitting cost $f_t(x) = \frac{1}{2}(x-v_t)^T A (x-v_t)$ and switching cost $c(x_t,x_{t-1}) = \frac{1}{2}\|x_t - x_{t-1}\|_2^2$ satisfies
 $\text{CR}_{\text{\ai{}}} \leq 1 + \frac{1}{\lambda_{\min}^A}.$
\end{theorem}

The above theorem is proved in Appendix \ref{proof:AI_cr}. Contrasting the adversarial performance of \ai{} with that of \robd{}, the cost ratio between Lazy Adaptive Interpolation and \robd{} can become unbounded as $\lambda^A_{\min}$ shrinks, in the following manner
\begin{equation*}
\frac{\text{Cost}_{\ai{}}[1,T]}{\text{Cost}_{\robd{}}[1,T]} \leq 
\frac{1 + \frac{1}{\lambda^A_{\min}}}{1 + \frac{1}{2}\left(\sqrt{1 + \frac{4}{\lambda_{\min}^A}} - 1\right)}
\end{equation*}
Although the aforementioned observation serves as an upper bound, we note that the suboptimal performance of \ai{} relative to \robd{} is further evident in our numerical experiments. 
The sub-optimality of \ai{} can be primarily attributed to the fact that it relies only on the matrix $C_T$, which further dictates the entire matrix sequence $\{C_t\}_{t=1}^T$ of \ai{}.

Note that the uniqueness of the online optimal algorithm establishes that any algorithm aiming for optimality in the stochastic setting behaves poorly in the adversarial setting.

We analyze the \ai{} algorithm in the adversarial context by developing a framework that can furnish adversarial guarantees for a large of class of static and dynamic algorithms within a broader family of cost functions. In that context, we briefly introduce some terminology associated with this framework, frequently encountered in the convex optimization literature. 

\begin{definition}
\textit{The Bregman divergence between two vectors in $\R^d$, $y$ and $x$, with respect to $h(\cdot)$. is defined as} $D_h(y||x):= h(y) - h(x) - \langle \nabla h(x), y-x \rangle.$ 
\textit{The function $h(x)$ is $\alpha$-strongly convex and $\beta$-smooth if and only if}
\begin{align}
    \frac{\alpha}{2}\|y-x\|_2^2\leq D_h(y||x) \leq \frac{\beta}{2}\|y-x\|_2^2 \text{ }\forall \text{ } y,x
\end{align}
\end{definition}

We now present a bound on the competitive ratio of a general set of policies, including \ai{}, for $m$-strongly convex hitting costs $f_t(x)$ and switching costs as $D_h(x_t||x_{t-1})$, where $h(\cdot)$ is $\alpha$-strongly convex and $\beta$-smooth. 

\begin{theorem}\label{thm:adv_analysis_framework}
Any online algorithm \alg{} that can be expressed in the following form, where $g_t(\cdot)$ is $\alpha_t'-$strongly convex and $\beta_t'-$smooth for $0 \leq \alpha'_t \leq \beta_t$,
\begin{equation*}
    x_t = \argmin_x f_t(x) + D_{h}(x||x_{t-1}) + D_{g_t}(x||v_t) \text{ } \forall \text{ } t.
\end{equation*}
has a competitive ratio in the adversarial setting upper-bounded as
\begin{equation*}
    \text{CR}_\alg{} \leq 1 + \max \left\{\frac{\max_t \beta'_t}{m} , \frac{\beta^2/ \alpha}{\min_t \alpha'_{t-1} + m}  \right\}
\end{equation*}
\end{theorem}

The significance of this result lies in its applicability to a broad spectrum of algorithms, all while preserving the optimal competitive ratio \cite{GoelLinWierman19} for squared switching costs, that is, $c(x_t,x_{t-1}) = \frac{1}{2}\norm{x_t - x_{t-1}}^2$. The proof, presented in Section \ref{proof:adv_analysis_framework}, follows a potential function technique, a popular approach in the adversarial online algorithms literature. The cornerstone of the proof is a potential function that is tailored to the algorithm considered here, yielding a competitive ratio that is specific to the $\{g_t(\cdot)\}_t$.

\subsection{A Near-Optimal Algorithm for Stochastic and Adversarial SOQO}\label{subsec:best of both}
In the previous two subsections, we demonstrated the shortcomings of the adversarial optimal algorithm in stochastic settings and the sub-optimal performance of the stochastic optimal algorithm in adversarial environments. Our aim now is \textit{to design an algorithm that has adversarial performance similar to that of \robd{} without sacrificing stochastic performance.} 

Our proposed algorithm builds on \ai{}, and is presented in Algorithm \ref{alg:laigamma}. \LAItxt{}'s design incorporates a parameter $\gamma$, such that when $\gamma = 0$, it behaves as Lazy Adaptive Interpolation, and when $\gamma = 1$, it becomes a fixed interpolation algorithm with $C = C_L$, which achieves a near-optimal competitive ratio. The expression in $\Tilde{C}_t$, specifically $ \frac{\lambda_i^A}{2}\left( \left(1 + \frac{4}{\lambda_i^A}\right)^{\frac{\gamma}{2}} - 1 \right) $, is chosen to replicate the structure of \robd{}, that is, $\frac{\lambda_{\min}^A}{2}\left( \left(1 + \frac{4}{\lambda_{\min}^A}\right)^{\frac{1}{2}} - 1\right)$, while being close to \ai{}. 

\begin{algorithm}[t]
\caption{Lazy Adaptive Interpolation $(\gamma)$}\label{alg:NOFI_AI}
\flushleft \textbf{Input:} $A\ (=P D_A P^T)$, where $D_A$ is a diagonal matrix, $\gamma \in [0,1]$
\flushleft \textbf{Initialize:} $\Tilde{\lambda}_i^T = \left(1 + \lambda_i^A + \frac{\lambda_i^A}{2}\left( \left(1 + \frac{4}{\lambda_i^A}\right)^{\frac{\gamma}{2}} - 1 \right)\right)^{-1}$ and $\Tilde{D}_t = diag\left(\Tilde{\lambda}_1^t, \ldots, \Tilde{\lambda}_d^t \right)$. Define $\Tilde{C}_T = P \Tilde{D}_T P^T$ and  $\Tilde{C}_{t}^{-1} = 2I + A - \Tilde{C}_{t+1}$ $\forall$ $t \in \{1,\ldots, T-1\}$
\begin{algorithmic}
\Procedure{\LAItxt{}}{}
\For{$t = 1,2,\ldots,T$}
\State $x_t \gets \Tilde{C}_t x_{t-1} + (I-\Tilde{C}_t) v_t$
\EndFor
\EndProcedure
\end{algorithmic}\label{alg:laigamma}
\end{algorithm}

Our choice of $\Tilde{C}_t$ is guided by two key observations from previous subsections. First, we recognize the significance of the recursion in \ai{}'s matrix sequence, which, as we will see, is instrumental in achieving near-optimal stochastic performance. Second, we take into account the dependence of \robd{} on $\lambda_{\min}^A$ in Definition \ref{def:robd}, and its effect on adversarial performance.

Our main result below characterizes \LAItxt{}'s performance in both adversarial and stochastic settings across the entire spectrum of $\gamma$ values.
\newpage
\begin{theorem}\label{thm:NOFI_AI_perf}The following two results hold for \LAItxt{} in the stochastic and adversarial settings:

\begin{enumerate} [{\normalfont(i)}]
\item In the stochastic setting \eqref{eqn:minimizer_martingale} with increments having identical covariance $\Sigma$, \LAItxt{} has \textbf{constant regret} with respect to \ai\, where $\sigma^2 = tr(\Sigma)$.
\begin{equation*}\label{thm:NOFI_AI_perf_CR}
\text{Regret}_{\LAImath{}}[1,T] \leq  \frac{\sigma^2}{4} \left( \frac{\left(1 + \frac{4}{\lambda_{\min}^A}\right)^{\frac{\gamma}{2}} - 1}{\lambda_{\min}^A + 2} \right).
\end{equation*}

\item In the adversarial setting, \LAItxt{}'s competitive ratio is upper bounded as
\begin{align*}\label{thm:NOFI_AI_perf_regret}
     \text{CR}_{\LAImath{}}\leq 1 + \max \Bigg\{\frac{1}{2}\left(\sqrt{\kappa(A)^2 +  \frac{4\kappa(A)}{\lambda_{\min}^A}} - \kappa(A) \right),\frac{2}{\lambda_{\min}^A}\left( \left(1 + \frac{4}{\lambda_{\min}^A}\right)^{\frac{\gamma}{2}} + 1 \right)^{-1} \Bigg\}
\end{align*}
where $\kappa(A)$ is the condition number of $A$. 
\end{enumerate}
\end{theorem}

The guarantee on adversarial performance, proved in Appendix \ref{proof:NOFI_AI_CR_proofs}, results from this specific choice of $\Tilde{C}_T$, coupled with the competitive analysis framework discussed in Theorem \ref{thm:adv_analysis_framework}. The proof of stochastic performance leverages the structure of the recursion $\Tilde{C}_{t}^{-1} = 2I + A - \Tilde{C}_{t+1}$ to show constant regret for any interpolation algorithm satisfying it. We show this in detail in Appendix \ref{proof:NOFI_AI_perf_stoch}.

Recall the importance of the case of $\lambda^A_i \ll 1$ discussed in Remark \ref{remark:low_lambda}. We study $\text{CR}_{\LAImath{}}$ in this regime and contrast it with the adversarial optimal \robd{} (in terms of dependence of the competitive ratio on $\lambda^A_i$). Below, we present our \textit{\textbf{best-of-both-worlds}} result, with the proof in Appendix \ref{proof:NOFI_AI_CR_proofs}.

\begin{corollary}\label{corr:LAI_best_of_both}
For the $\lambda_{\max}^A \ll 1$ regime in the adversarial setting, the \LAItxt{} algorithm with $\gamma = 1$ achieves a competitive ratio of
\begin{equation*}
    \text{CR}_\nofi{} \leq 1 + \sqrt{\frac{\kappa(A)}{\lambda_{\min}^A}}
\end{equation*}
while simultaneously achieving a \textbf{constant regret} in the stochastic setting \eqref{eqn:minimizer_martingale}.
\end{corollary}
Comparing the above result with \robd{}'s competitive ratio of $O(1/\sqrt{\lambda_{\min}^A})$ in the $\lambda_{\min}^A \ll 1$ regime highlights that \nofi{} achieves a near-optimal competitive ratio in the adversarial setting, for well-behaved hitting costs, while simultaneously having near-optimal performance in the martingale setting.

\section{Numerical Experiments}\label{sec:num_exps}
To further explore SOQO in stochastic and adversarial settings, we conduct empirical experiments to evaluate the performance of our algorithms in a range of environments. Our experiments fall into two main categories: first, purely stochastic experiments where we provide empirical evidence for two primary claims – the inferior performance of \robd{} in comparison to \LAItxt{} and \ai{}, and the robustness of our algorithms to distribution shifts and variations in tail behavior (light or heavy). Second, we delve into mixed adversarial-stochastic experiments that combine elements of both adversarial and stochastic scenarios, examining how our algorithms perform in this hybrid setting. Although we offer a summarized overview of our experimental setup here, detailed procedures can be found in Appendix \ref{numerical_exp_details}. 

In all of our experiments, we maintain a consistent action space $\R^{10}$ and employ the matrix $A$ with eigenvalues selected from one of three sequences: $\left\{0.3^i \right\}_{i=0}^9$, $\left\{0.45^i\right\}_{i=0}^9$, or $\left\{0.5^i\right\}_{i=0}^9$. The sequence of minimizers, denoted $\{v_t\}_t$, is adjusted according to the specific type of experiment under consideration. In both stochastic and adversarial environments, we compare the \robd{} algorithm to the \LAItxt{} algorithm with $\gamma = 1$. The specific value of $\gamma$ is chosen to emphasize that even the adversarial extreme of \LAItxt{} demonstrates exceptional stochastic performance.

\subsection{Experiments in purely stochastic environments}
Our stochastic experiments are categorized into three distinct groups. The first category encompasses light tail distributions, featuring five zero-mean distributions: uniform, normal, Laplace, logistic, and Gumbel. In this particular experiment, we introduce a new distribution every $T/5$ rounds, underscoring the robustness of our results to distribution shifts. The following two categories focus on heavy-tail distributions, specifically log-normal and Pareto distributions. Across all these experiments, the minimizer sequence adheres to a martingale structure, where the increments exhibit uniform variance and are drawn from the respective distribution. Our analysis involves plotting the expected total regret (relative to \ai{}) for both $\left\{0.3^i \right\}_{i=0}^9$ and $\left\{0.5^i \right\}_{i=0}^9$ over various horizon values $T \in \{1,2,\ldots,100\}$. We present the results as the sample mean derived from $N=1000$ runs, accompanied by the $95^{th}$ percentile for added clarity.

The trends in Fig. \ref{fig:stoch_sims} validate our claims in Theorem \ref{thm:R-OBD_lin_regret} and Theorem \ref{thm:NOFI_AI_perf} (i) regarding \robd{}'s linear regret and \nofi{}'s constant regret. In fact, we observe that the regret of \nofi{} is virtually zero in contrast to that of \robd{}, demonstrating the superiority of \LAItxt{} in practice. The insensitivity of \LAItxt{} to the form of the distribution is further highlighted by its consistent near-optimal stochastic performance in all simulated stochastic settings. In particular, we would like to stress the stability demonstrated by \LAItxt{} under shifting distributions and in heavy-tailed stochastic environments. The negligible regret shown by \nofi{} establishes the superiority of \LAItxt{} over \robd{} in any distribution with finite variance.

\begin{figure}[t!p]
    \centering
    \begin{subfigure}[b]{0.3\textwidth}
        \centering
        \includegraphics[height=1.35in]{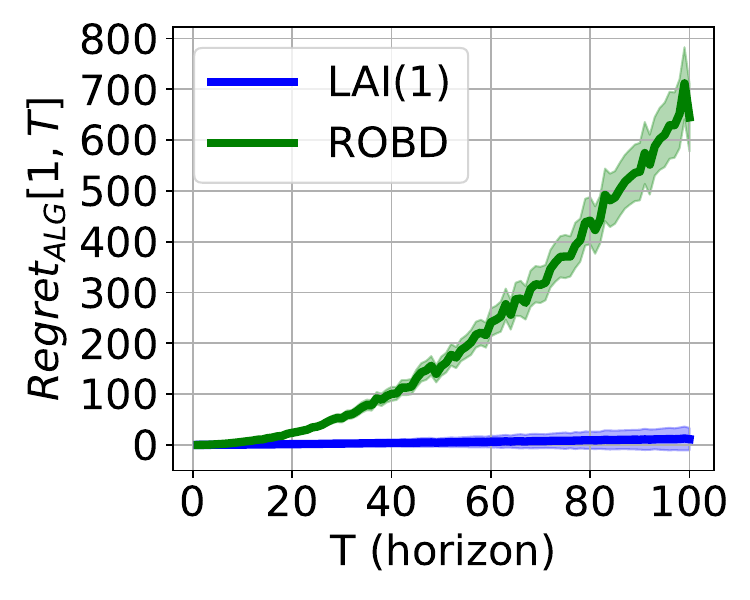}
        \caption{light tail with $\lambda_i^A = 0.3^{i-1}$}
        \label{fig:stoch_sim_0.3_light}
    \end{subfigure}
    \begin{subfigure}[b]{0.3\textwidth}
        \centering
        \includegraphics[height=1.35in]{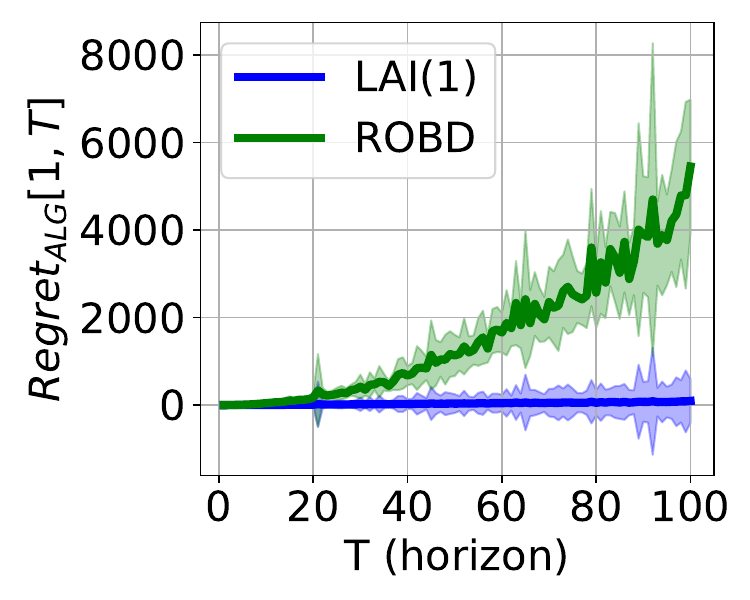}
        \caption{log-normal with $\lambda_i^A = 0.3^{i-1}$}
        \label{fig:stoch_sim_0.3_lognorm}
    \end{subfigure}
    \begin{subfigure}[b]{0.3\textwidth}
        \centering
        \includegraphics[height=1.35in]{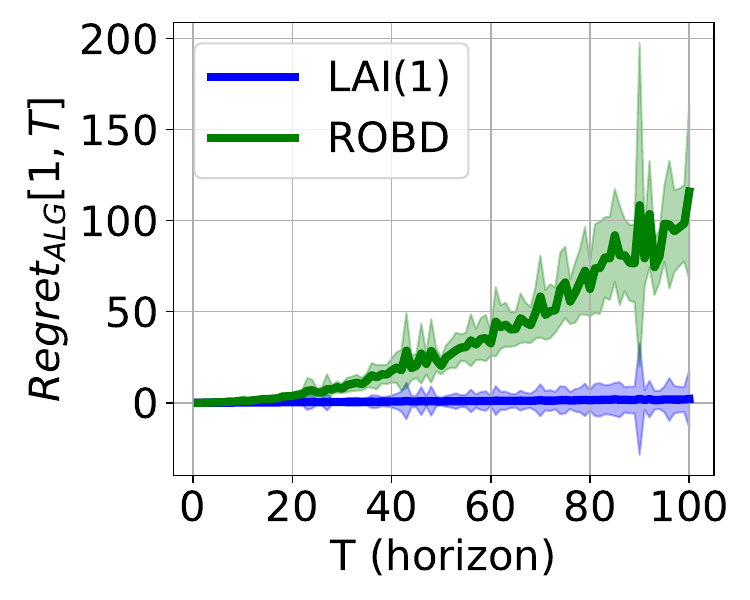}
        \caption{pareto with $\lambda_i^A = 0.3^{i-1}$}
        \label{fig:stoch_sim_0.3_pareto}
    \end{subfigure}
    
    \begin{subfigure}[b]{0.3\textwidth}
        \centering
        \includegraphics[height=1.35in]{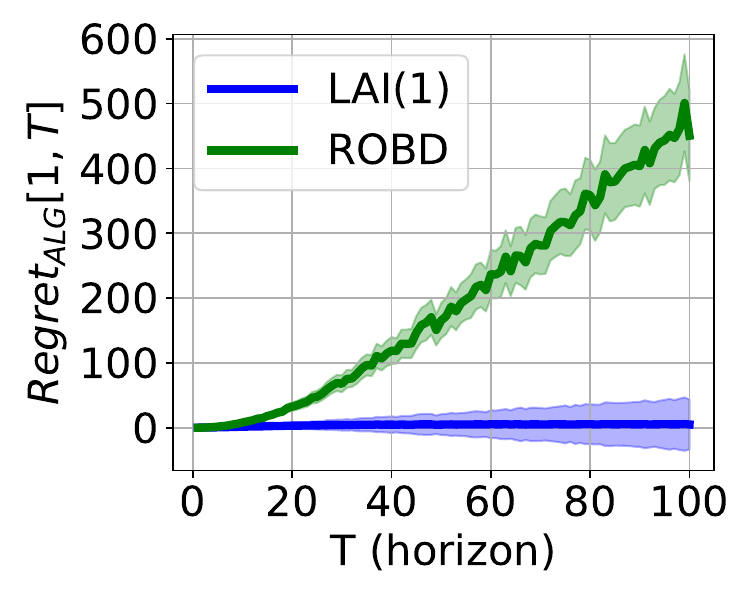}
        \caption{light tail with $\lambda_i^A = 0.5^{i-1}$}
        \label{fig:stoch_sim_0.5_light}
    \end{subfigure}
    \begin{subfigure}[b]{0.3\textwidth}
        \centering
        \includegraphics[height=1.35in]{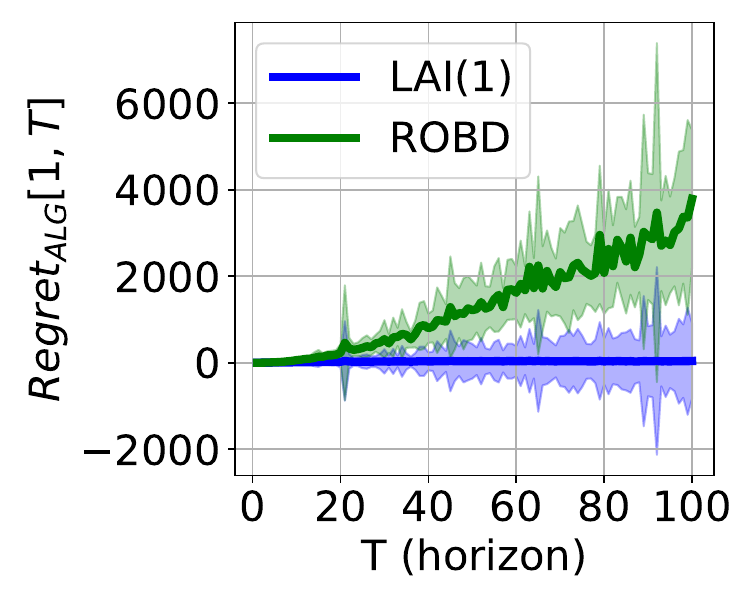}
        \caption{log-normal with $\lambda_i^A = 0.5^{i-1}$}
        \label{fig:stoch_sim_0.5_lognorm}
    \end{subfigure}
    \begin{subfigure}[b]{0.3\textwidth}
        \centering
        \includegraphics[height=1.35in]{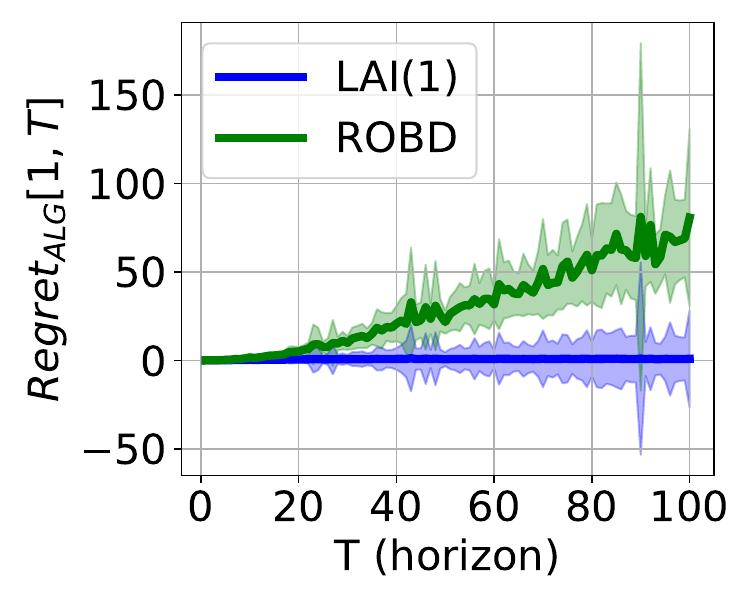}
        \caption{pareto with $\lambda_i^A = 0.5^{i-1}$}
        \label{fig:stoch_sim_0.5_pareto}
    \end{subfigure}
    \caption{Regret of \nofi{} and \robd{} for martingale minimizers with light and heavy tails}
    \label{fig:stoch_sims}
\end{figure}

\subsection{Experiments in stochastic and adversarial environments}
In this series of experiments, we introduce adversarial minimizers into a martingale minimizer sequence, with the extent of adversarial influence determined by the parameter known as the \textit{adversarial percentage} denoted $(p)$. This parameter spans from $0$ (indicating a fully stochastic scenario) to $100$ (representing a fully adversarial scenario). To facilitate meaningful comparisons across stochastic and adversarial environments, we calculated the ratio of an online algorithm's total cost to that of \ai{}. This normalization technique helps account for the differences in the orders of magnitude between the total costs in stochastic and adversarial settings, allowing us to evaluate the relative performance of various algorithms effectively.

In these experiments, we explore scenarios involving three different sequences of eigenvalues: $\left\{0.3^i \right\}_{i=0}^9$, $\left\{0.45^i \right\}_{i=0}^9$, and $\left\{0.5^i \right\}_{i=0}^9$. The key observation is that, on the adversarial end, the relative performance of \nofi{} and \robd{} is consistent across different $A$ matrices. However, it becomes evident that the stochastic performance of \robd{} deteriorates significantly when a smaller $A$ matrix is used. 

As we gradually intensify the adversarial characteristics of the environment, we notice a relatively smooth shift from \nofi{} to \robd{} in terms of identifying the ``superior algorithm." In fact, until a certain threshold of adversarial influence, approximately around $20$\%, \nofi{} surpasses \robd{} in performance. This intriguing observation prompts a more in-depth analysis of SOQO in a stochastic environment with adversarial contamination.

\subsection{Experiments in stochastic and adversarial environments}
In this series of experiments, we introduce adversarial minimizers into a martingale minimizer sequence, with the extent of adversarial influence determined by the parameter known as the \textit{adversarial percentage} denoted $(p)$. This parameter spans from $0$ (indicating a fully stochastic scenario) to $100$ (representing a fully adversarial scenario). To facilitate meaningful comparisons across stochastic and adversarial environments, we calculated the ratio of an online algorithm's total cost to that of \ai{}. This normalization technique helps account for the differences in the orders of magnitude between the total costs in stochastic and adversarial settings, allowing us to evaluate the relative performance of various algorithms effectively.

\begin{figure}[t!p]
    \centering
    \begin{subfigure}[b]{0.3\textwidth}
        \centering
        \includegraphics[height=1.35in]{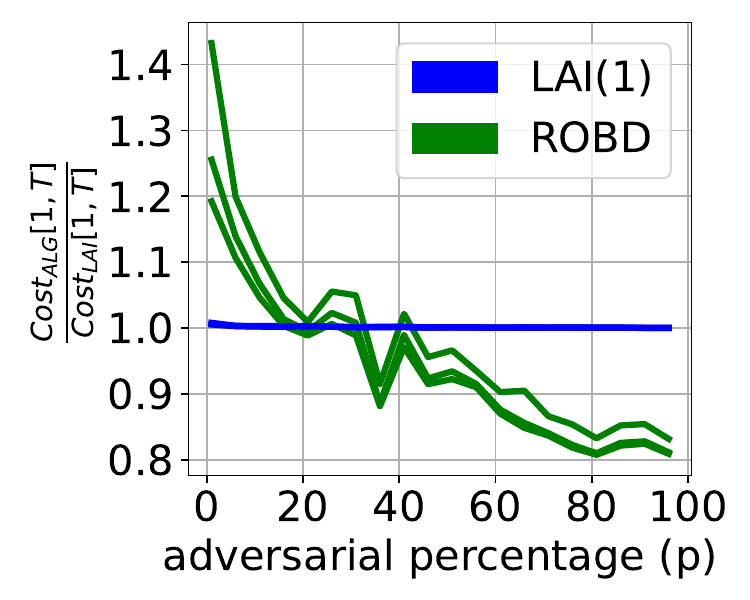}
        \caption{normal to adversarial}
        \label{fig:adv_sim_light}
    \end{subfigure}
    \begin{subfigure}[b]{0.3\textwidth}
        \centering
        \includegraphics[height=1.35in]{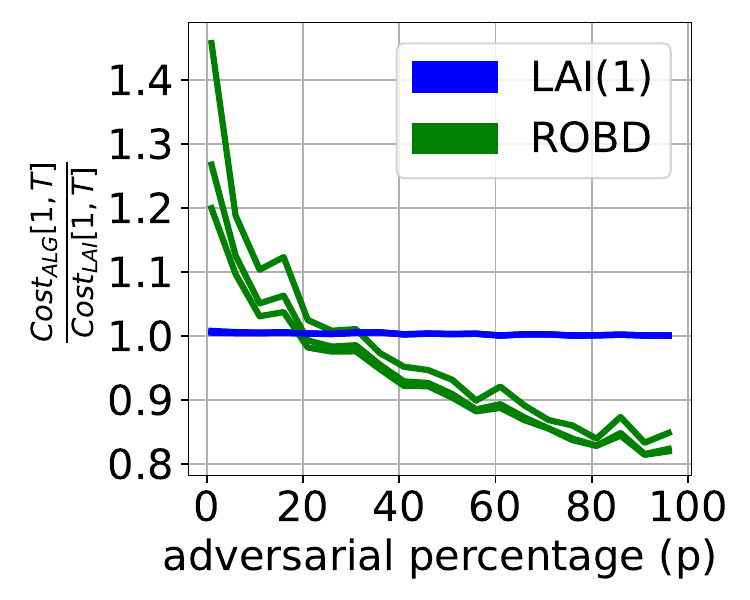}
        \caption{log-normal to adversarial}
        \label{fig:adv_sim_lognorm}
    \end{subfigure}
    \begin{subfigure}[b]{0.3\textwidth}
        \centering
        \includegraphics[height=1.35in]{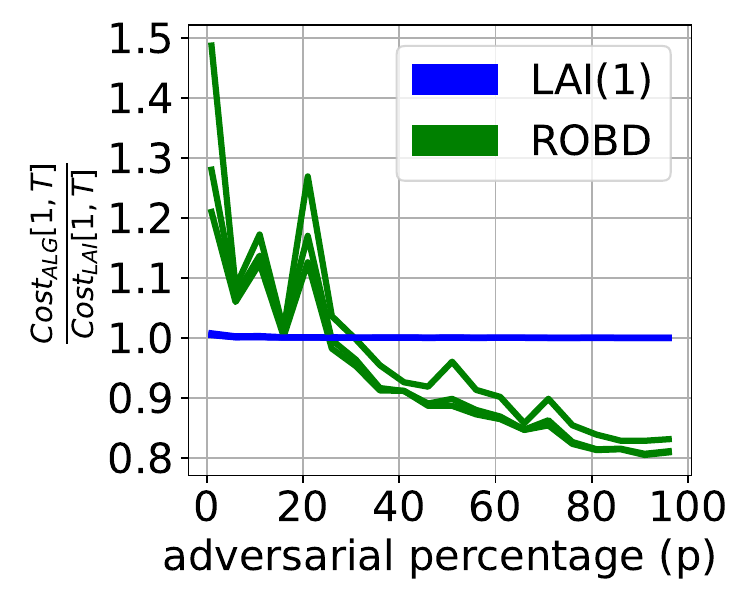}
        \caption{pareto to adversarial}
        \label{fig:adv_sim_pareto}
    \end{subfigure}
    \caption{Behavior of \nofi{} and \robd{} for a mixed sequence of minimizers. In each figure, the topmost plot corresponds to $\left\{0.3^i \right\}_{i=0}^9$, the middle one represents $\left\{0.45^i \right\}_{i=0}^9$, and the bottom-most plot pertains to $\left\{0.5^i \right\}_{i=0}^9$.}
    \label{fig:adv_sims}
\end{figure}
\section{Concluding Remarks}
This research broadens the horizons of smoothed online optimization in two unexplored dimensions: the examination of SOQO within stochastic contexts and the concept of ``best of both worlds" algorithms in this domain. These directions hold immense potential, especially the stochastic analysis beyond quadratic costs. Within the framework discussed in this work, two extensions are especially intriguing: (i) multi-agents systems with costs coupling the agents and (ii) learning the $A$ matrix.

\bibliography{references}
\bibliographystyle{acm}

\newpage
\begin{appendices}
\section{Detailed Literature Review}\label{sec:lit_review}

\subsection{Stochastic Smoothed Online Optimization}
It is noteworthy that the literature in smoothed online optimization focuses on achieving strong performance in adversarial environments, without addressing the possibility of a stochastic environment, let alone seeking best-of-both-worlds algorithms. In the LQR control literature, there have been some works that consider noisy inputs \cite{ChenWierman15,ChenWierman16}. However, their noise model and dynamics are very specific, while our set-up considers stochasticity in more generality. In other problems involving dynamical systems, stochastic inputs have been considered, with prominent example being staff/inventory management \cite{Gandhi2010,Jennings1996} and power systems \cite{Urgaonkar2011,Koutsopoulos2011}.

In particular, within the realm of multi-arm bandits (MAB), \cite{BS12} presents an algorithm that provides robust performance guarantees for both Independent and Identically Distributed (IID) reward sequences as well as adversarial reward sequences. In the context of Online Convex Optimization (OCO) literature, recent attention~\cite{ChenTu23,Sachs22,Sachs23} has been directed towards bridging the gap between an IID environment, which is relevant to stochastic (stationary) optimization, and an adversarial environment, representing the traditional OCO problem. 
However, it is important to note that these works \textit{do not address switching costs}, and the assumptions regarding hitting costs and action spaces differ between OCO and ours.

\subsection{Adversarial Smoothed Online Optimization}
The body of research on smoothed online optimization is predominantly centered around adversarial results and can be broadly divided into two distinct communities. The first community examines it through the lens of Metrical Task Systems and employs the \textit{competitive ratio} as the performance metric. Works of \cite{LinWiermanAndrew11, BansalGupta15} gave the first results considering the one-dimensional case. Further research \cite{ChenWierman18} established that assumptions stronger than convexity, like $\alpha$-polyhedrality or $m-$strong convexity are required for meaningful results. \cite{ZhangYang21} proved for $\alpha$-polyhedral hitting costs and $\ell_1$-norm switching costs that following the minimizer has a very good competitive ratio of $\left(\max\left\{ 1,\frac{2}{\alpha} \right\}\right)$. \cite{GoelWierman19} gave an order-optimal algorithm \robd{} for $m-$strongly convex hitting costs and $\ell_2$-norm switching costs with a competitive ratio of $1 + \mathcal{O}\big(\frac{1}{\sqrt{m}}\big)$. The analytical tools present in these works, like potential function analysis and receding horizon control, focus on guaranteeing worst-case performance, and do not take advantage of structure in the hitting costs.

Apart from the competitive ratio, regret with respect to the dynamic, hindsight optimal sequence of decisions has been another important performance metric, popular in the OCO literature.
In the adversarial setting, \cite{AndrewWierman13} showed that for any online algorithm, it is impossible to simultaneously achieve a finite competitive ratio and a sublinear regret. 
Subsequent research aimed to establish a dynamic regret that scales in relation to the path-length, denoted as $L_T$, of the optimal action sequence. For scenarios involving $\alpha$-polyhedral convex hitting costs and $\ell_2$-norm switching costs, \cite{ChenWierman18} demonstrated a dynamic regret of $\mathcal{O}\left( \sqrt{T L_T} \right)$ for a modified version of OBD. 
In cases with strongly convex hitting costs and squared $\ell_2$-norm switching costs, \cite{LiGuannan18} established a lower bound of $\mathcal{O}\left(L_T\right)$ on the dynamic regret. 
This bound was closely matched by \cite{GoelLinWierman19}. Finally, \cite{GoelLinWierman19} demonstrated that \robd{} achieves a dynamic regret of $\mathcal{O}\left( \sqrt{T L_T} \right)$ in these settings. The problem set-up in these results is different than the MTS-style ones, and there is very little on their connection \cite{AndrewWierman13}.

It is worth noting that these dynamic regret results rely on the strong assumption that the action space has a bounded diameter. Since we are dealing with metric switching costs, the absence of a boundedness assumption makes the problem much more difficult. Considering we have a stochastic setting in this work, the absence of such assumptions generalizes our results to any distribution over $\R^d$ with finite second moment (light or heavy tail), something one cannot achieve in the OCO setting.

\section{Main Proofs} \label{sec:theorem proofs}
In this section, we shed light on the analysis techniques employed in this work, through detailed proofs of some of our main contributions. These include the stochastic optimality of \ai{}, the poor performance of \robd{} algorithm, the near-optimality of \LAItxt{} in both stochastic and adversarial settings, and our general adversarial analysis framework.

\subsection{Proof of Theorem \ref{thm:LAI_total_cost}}\label{proof:LAI_perf}
To prove these results, we show that \ai{} is a dynamic programming solution in our stochastic setting \eqref{eqn:minimizer_martingale}; thus proving it to be an optimal online algorithm. The first step is to show the optimality of the action at round $T$, given by the following lemma.

\begin{lemma}\label{lemma: val_fn_T}
The online optimal action at round $T$ is
\begin{align*}
    x^{\ai{}}_T = C_T x^{\ai{}}_{T-1} + \left(I - C_T \right) v_T
\end{align*}
and, consequently, the value function \eqref{eqn:value_fn_defn} for round $T$ is
\begin{align*}
    V_T\left(x^{\ai{}}_{T-1},v_T \right) = \frac{1}{2}\left(x^{\ai{}}_{T-1} - v_T \right)^T \left(I - C_T \right) \left(x^{\ai{}}_{T-1} - v_T \right)
\end{align*}
\end{lemma}
At round $T$, the optimal decision is
\begin{align}
    x_T = \argmin_x \frac{1}{2}(x-v_T)^T A (x-v_T) + \frac{1}{2}\|x - x^{\ai{}}_{T-1}\|_2^2.
\end{align}
\begin{proof}
Differentiating the objective,
\begin{align}
    A(x-v_T) + (x-x^{\ai{}}_{T-1}) = 0
\end{align}
gives the online optimal action at round $T$ as
\begin{align}
    x^{\ai{}}_T &= (A+I)^{-1} x^{\ai{}}_{T-1} + (A+I)^{-1} A v_T\\
    &= (A+I)^{-1} x^{\ai{}}_{T-1} + (I - (A+I)^{-1}) v_T\\
    &= C_T x^{\ai{}}_{T-1} + (I-C_T)v_T.
\end{align}
Putting $x^{\ai{}}_T$ back into the current cost gives the value function at $T$,
\begin{align}
\begin{split}
V_{T}(x^{\ai{}}_{T-1}, v_T) =& \frac{1}{2}(x_T - v_T)^T A (x_T - v_T) +\frac{1}{2}\|x_T - x_{T-1}\|_2^2
\end{split}\\
\begin{split}
=& \frac{1}{2} (x^{\ai{}}_{T-1} - v_T)^T \left((A+I)^{-1} A (A+I)^{-1} \right) (x^{\ai{}}_{T-1} - v_T)\\
&+ \frac{1}{2} (x^{\ai{}}_{T-1} - v_T)^T \left( (A+I)^{-1} A^2 (A+I)^{-1} \right) (x^{\ai{}}_{T-1} - v_T)
\end{split}\\
=& \frac{1}{2} (x^{\ai{}}_{T-1} - v_T)^T \left(  (A+I)^{-1} A \right) (x^{\ai{}}_{T-1} - v_T)\\
=& \frac{1}{2} (x^{\ai{}}_{T-1} - v_T)^T \left(  I - C_T \right) (x^{\ai{}}_{T-1} - v_T) \label{AI_final_round_cost}
\end{align}
\end{proof}
We next present the main element of the proof of \ai{}'s optimality in the form of the following proposition, from which Theorem \ref{thm:LAI_total_cost} directly follows.

\begin{proposition}\label{prop:AI_val_fn}
Consider hitting costs $\frac{1}{2} (x-v_t)^T A (x - v_t)$ and switching cost is $\frac{1}{2}\|x_t - x_{t-1}\|_2^2$ with the sequence of minimizers being a martingale, that is \eqref{eqn:minimizer_martingale}. The exact value function at $T-t$ is
\begin{equation*}
\begin{split}
    V_{T-t}(x^{\ai{}}_{T-t-1}, v_{T-t}) =& \frac{1}{2}(x^{\ai{}}_{T-t-1}-v_{T-t})^T (I - C_{T-t}) (x^{\ai{}}_{T-t-1}-v_{T-t})\\
    &+ \sum_{s=0}^{t-1} \frac{1}{2} \E\left[(v_{T-s} - v_{T-s-1})^T (I-C_{T-s}) (v_{T-s} - v_{T-s-1}) | \mathcal{F}_{T-t} \right]
\end{split}
\end{equation*}
\end{proposition}
\begin{proof}

The proof follows via an induction argument. The value function at round $T$, according to Lemma \ref{lemma: val_fn_T}, is
\begin{align}
    V_{T}(x^{\ai{}}_{T-1}, v_{T}) = \frac{1}{2} (x^{\ai{}}_{T-1} - v_T)^T \left(  I - C_T \right) (x^{\ai{}}_{T-1} - v_T).
\end{align}
The proposition is, therefore, true for $t=0$ (sum from $0$ to $-1$ is zero). Assuming the proposition for some $t\geq 0$, we have the optimal action at round $T-t-1$ as
\begin{align}
\begin{split}
x^{\ai{}}_{T-t-1} = \argmin_x & \bigg\{\frac{1}{2} (x - v_{T-t-1})^T A (x - v_{T-t-1}) + \frac{1}{2}\|x - x^{\ai{}}_{T-t-2}\|_2^2\\
&+ \E[V_{T-t}(x,v_{T-t}) | \mathcal{F}_{T-t-1}]\bigg\}
\end{split}\\
\begin{split}
=\argmin_x & \bigg\{\frac{1}{2} (x - v_{T-t-1})^T A (x - v_{T-t-1}) + \frac{1}{2}\|x - x^{\ai{}}_{T-t-2}\|_2^2\\
&+ \frac{1}{2} \E\left[(x-v_{T-t})^T (I - C_{T-t}) (x-v_{T-t}) | \mathcal{F}_{T-t-1}\right]\\
&+ \sum_{s=0}^{t-1} \frac{1}{2}\E\left[(v_{T-s} - v_{T-s-1})^T (I-C_{T-s}) (v_{T-s} - v_{T-s-1}) | \mathcal{F}_{T-t-1} \right]\bigg\}
\end{split}\\
\begin{split}
= \argmin_x & \bigg\{\frac{1}{2} (x - v_{T-t-1})^T (A + I -C_{T-t}) (x - v_{T-t-1}) + \frac{1}{2}\|x - x^{\ai{}}_{T-t-2}\|_2^2\\
&+ (x - v_{T-t-1})^T (I -C_{T-t}) \E[v_{T-t-1} - v_{T-t} | \mathcal{F}_{T-t-1}]\\
&+ \frac{1}{2}\E[(v_{T-t} - v_{T-t-1})^T (I-C_{T-t})(v_{T-t} - v_{T-t-1}) | \mathcal{F}_{T-t-1}]\\
&+ \sum_{s=0}^{t-1} \frac{1}{2} \E\left[ (v_{T-s} - v_{T-s-1})^T (I-C_{T-s}) (v_{T-s} - v_{T-s-1}) | \mathcal{F}_{T-t-1} \right] \bigg\}
\end{split}\\
\begin{split}
=\argmin_x & \bigg\{\frac{1}{2} (x - v_{T-t-1})^T (A + I -C_{T-t}) (x - v_{T-t-1}) + \frac{1}{2}\|x - x^{\ai{}}_{T-t-2}\|_2^2\\
&+ \sum_{s=0}^{t} \frac{1}{2} \E\left[ (v_{T-s} - v_{T-s-1})^T (I-C_{T-s}) (v_{T-s} - v_{T-s-1}) | \mathcal{F}_{T-t-1} \right]\bigg\}
\end{split}
\end{align}
Differentiating the objective of the optimization problem above, we get
\begin{align}
    (A + I - C_{T-t})(x - v_{T-t-1}) + (x - x^{\ai{}}_{T-t-2}) = 0
\end{align}
giving us the optimal action at round $T-t-1$ as
\begin{align}
    x^{\ai{}}_{T-t-1} &= (2I + A - C_{T-t})^{-1} x^{\ai{}}_{T-t-2} + \left(I - (2I + A - C_{T-t})^{-1} \right)v_{T-t-1}\\
    &= C_{T-t-1} x^{\ai{}}_{T-t-2} + \left(I - C_{T-t-1} \right) v_{T-t-1}
\end{align}
where the invertibility of $(2I + A - C_{T-t})$ is easy to prove and is provided in the appendix. Consequently, the value function at round $T-t-1$ as
\begin{align}
\begin{split}
V_{T-t-1}(x^{\ai{}}_{T-t-2}&, v_{T-t-1})\\
=& \frac{1}{2} (x^{\ai{}}_{T-t-1} - v_{T-t-1})^T A (x^{\ai{}}_{T-t-1} - v_{T-t-1}) + \frac{1}{2}\|x^{\ai{}}_{T-t-1} - x^{\ai{}}_{T-t-2}\|_2^2\\
&+ \E[V_{T-t}(x^{\ai{}}_{T-t-1}, v_{T-t}) | \mathcal{F}_{T-t-1}]
\end{split}\\
\begin{split}
=& \frac{1}{2} (x^{\ai{}}_{T-t-2} - v_{T-t-1})^T C_{T-t-1}(C_{T-t-1}^{-1} - I) C_{T-t-1}(x^{\ai{}}_{T-t-2} - v_{T-t-1})\\
&+ \frac{1}{2} (x^{\ai{}}_{T-t-2} - v_{T-t-1})^T(I-C_{T-t-1})^2 (x^{\ai{}}_{T-t-2} - v_{T-t-1})\\
&+ \sum_{s=0}^{t} \frac{1}{2} \E\left[ (v_{T-s} - v_{T-s-1})^T (I-C_{T-s}) (v_{T-s} - v_{T-s-1}) | \mathcal{F}_{T-t-1} \right]
\end{split}\\
\begin{split}
=& \frac{1}{2} (x^{\ai{}}_{T-t-2} - v_{T-t-1})^T (I - C_{T-t-1})(x^{\ai{}}_{T-t-2} - v_{T-t-1})\\
&+ \sum_{s=0}^{t} \frac{1}{2} \E\left[ (v_{T-s} - v_{T-s-1})^T (I-C_{T-s}) (v_{T-s} - v_{T-s-1}) | \mathcal{F}_{T-t-1} \right]
\end{split}
\end{align}
proving the proposition through induction. The cost of \ai{} algorithm, which is also the online optimal cost, is
\begin{align}
    \E&\left[\text{Cost}_{\text{\ai{}}}[1,T]\right] = \E[V_{1}(x_0,v_1)|\mathcal{F}_0] = \sum_{t=1}^{T} \frac{1}{2} \E\left[ (v_t - v_{t-1})^T (I-C_{t}) (v_{t} - v_{t-1}) | \mathcal{F}_{0} \right]
\end{align}
where $\mathcal{F}_0$ is the trivial sigma field and $v_0 = x_0$. The upper bound in terms of $C_L$ is a consequence of the following important observation.
\begin{lemma}
The matrix sequence in \ai{} algorithm, that is, $\{C_t\}_t$ and $C_L$ are related as
\begin{align*}
    C_L \prec C_1 \prec C_2 \prec \ldots \prec C_{T-1} \prec C_T.
\end{align*}
\end{lemma}
The proof follows from the recursion satisfied by $\{C_t\}_t$, that is, $C_{t-1}^{-1} = 2I + A - C_t$ and the fact that $C_L$ is the stationary point of this recursion. We prove this lemma, along with other interesting properties of $\{C_t\}_{t=1}^T$, in Corollary \ref{corr:Ct_add_props}.
\end{proof}

\subsection{Proof of Theorem \ref{thm:FI_lin_regret_short}}\label{proof:FI_lin_regret_short}
We split the proof into four parts, one of which will present the exact characterization of the lower bound. The first step is the cost of the \FI{} algorithm, as the stated in the proposition below.
\begin{proposition}\label{prop:FI_cost_gen}
The following is the expected cost of the \FI{} algorithm, with matrix $C$ for martingale minimizers \eqref{eqn:minimizer_martingale} with increments having same covariance matrix, $\Sigma$,
\begin{equation*}
\E[\text{Cost}_{\FI{}}[1,T]] \geq \sum_{t=1}^{T} \frac{1}{2}\E\left[(v_{t}-v_{t-1})^T\left(I-C^2\right)^{-1}(AC^2 + \left(I-C\right)^2) (v_{t}-v_{t-1}) \right] - \left(\frac{\left(\lambda_{\max}^{C} \right)^{2}}{1 - \left(\lambda_{\max}^{C} \right)^{2}} \right) \frac{\sigma^2}{2}
\end{equation*}
\end{proposition}
The proof of this involves an induction argument to get the cost of the last $t$ rounds and is presented right after the theorem proof in Appendix \ref{proof:FI_cost_gen} to focus on the regret lower bound. The next step is to compare it to the cost of the \ai{} algorithm to get the following lower bound to the regret.

\begin{theorem}\label{thm:FI_lin_regret}
Consider the hitting costs, switching costs and stochastic setting \eqref{eqn:minimizer_martingale} from the problem set-up with $(v_t - v_{t-1})$ having covariance matrix $\Sigma$. The following is the regret of the \FI{} algorithm,
\begin{align*}
\begin{split}
\E[Regret_{\FI{}}[1,T]]
\geq \frac{T}{2} \sum_{j \not \in Z} \left(w\left(\lambda_j^A, \lambda_j^{C}\right) - \left(1-\lambda_j^L\right)\right)(e_j^T P^T \Sigma P e_j) - \left(\frac{\left(\lambda_{\max}^{C} \right)^{2}}{1 - \left(\lambda_{\max}^{C} \right)^{2}} \right) \frac{\sigma^2}{2}.
\end{split}
\end{align*}
where $w(\alpha,c) = \frac{\alpha c^2 + (1-c)^2}{1-c^2}$, $Z = \{ i \in \{1,\ldots,d\}: \lambda_i^C = \lambda_{i}^L \}$, $P$ is the modal matrix of $A$ and, $\lambda_j^A, \lambda_j^C, \lambda_j^L$ are the $j^{th}$ eigenvalues of $A, C, C_L$ respectively.
\end{theorem}

\begin{proof}
We can write $A = PD_A P^T$, $C_L = P D_L P^T$ and $C = P D P^T$, where $P$ is the modal matrix of A and the middle matrix in each being the diagonal matrix of respective eigenvalues. 
\begin{align}
\begin{split}
\E[&\text{Cost}_{\FI{}}[1,T]] - \E[\text{Cost}_{\text{\ai{}}}[1,T]]\\
\geq & \sum_{t=1}^{T} \frac{1}{2}\E\left[(v_{t}-v_{t-1})^T\left(I-C^2\right)^{-1}(AC^2 + \left(I-C\right)^2) (v_{t}-v_{t-1}) \right]\\
&- \left(\frac{\left(\lambda_{\max}^{C} \right)^{2}}{1 - \left(\lambda_{\max}^{C} \right)^{2}} \right) \frac{\sigma^2}{2} - \sum_{t=1}^T \E[(v_t - v_{t-1})^T(1-C_L)(v_t - v_{t-1})]
\end{split}\\
\begin{split}
\geq & \sum_{t=1}^{T} \frac{1}{2}\E\left[(v_{t}-v_{t-1})^T P \left( \left(I-D^2\right)^{-1}(D_A D^2 + \left(I-D\right)^2) - (I - D_L)\right) P^T (v_{t}-v_{t-1}) \right]\\
&- \left(\frac{\left(\lambda_{\max}^{C} \right)^{2}}{1 - \left(\lambda_{\max}^{C} \right)^{2}} \right) \frac{\sigma^2}{2}
\end{split}
\end{align}
Observe that we can write
\begin{align}
     \left(I-D^2\right)^{-1}&(D_A D^2 + \left(I-D\right)^2) - (I - D_L)\\ &= diag\left(w\left(\lambda_1^A, \lambda_1^{C}\right) - \left(1-\lambda_1^L\right), \ldots, w\left(\lambda_d^A, \lambda_d^{C}\right) - \left(1-\lambda_d^L\right)\right)
\end{align}
We define the set of indices $Z = \{ i \in \{1,\ldots,d\}: \lambda_i^C = \lambda_{i}^L \}$. 
The following lemma shows that $\left(w\left(\lambda_j^A, \lambda_j^{C}\right) - \left(1-\lambda_j^L\right)\right) > 0 \text{ } \forall \text{ } j \not \in Z.$ 
\begin{lemma}\label{lemma:g function defn}
The minimum of $w(\alpha,c)$ has the following properties
\begin{gather*}
    \argmin_{c \in [0,1]} w(\alpha,c) = \frac{\alpha +2 - \sqrt{\alpha^2 + 4\alpha}}{2} = \frac{2}{\alpha +2 + \sqrt{\alpha^2 + 4\alpha}} = c_L\ \text{ and }\ 
    \min_{c \in [0,1]} w(\alpha,c) = 1- c_L.
\end{gather*}
\end{lemma}
Now, using Lemma \ref{lemma:g function defn},
$
    w\left(\lambda_i^A, \lambda_i^{C}\right) - \left(1-\lambda_i^L\right) > 0 \text{ if } i \not \in Z \text{ and } w\left(\lambda_i^A, \lambda_i^{C}\right) - \left(1-\lambda_i^L\right) = 0 \text { if } i \in Z.
$
There has to be $i$ such that $i \not \in Z$. Otherwise, $C$ and $C_L$ will have all eigenvalues and their corresponding eigenvectors same, making them equal.
Denoting $P^T(v_t - v_{t-1}) = y_t$, the gap of \FI{} to \ai{} is
\begin{align}
\begin{split}
Regret_{\FI{}}[1,T]\geq & \sum_{t=1}^{T} \frac{1}{2}\E\left[y_t^T diag\left(w\left(\lambda_1^A, \lambda_1^{C}\right) - \left(1-\lambda_1^L\right), \ldots, w\left(\lambda_d^A, \lambda_d^{C}\right) - \left(1-\lambda_d^L\right)\right) y_t \right]\\
&- \left(\frac{\left(\lambda_{\max}^{C} \right)^{2}}{1 - \left(\lambda_{\max}^{C} \right)^{2}} \right) \frac{\sigma^2}{2}
\end{split}\\
\begin{split}
= & \sum_{t=1}^{T} \frac{1}{2} \sum_{j \not \in Z} \left(w\left(\lambda_j^A, \lambda_j^{C}\right) - \left(1-\lambda_j^L\right)\right)\E(y_t)_j^2 - \left(\frac{\left(\lambda_{\max}^{C} \right)^{2}}{1 - \left(\lambda_{\max}^{C} \right)^{2}} \right) \frac{\sigma^2}{2}
\end{split}\\
\begin{split}
= & \frac{T}{2} \sum_{j \not \in Z} \left(w\left(\lambda_j^A, \lambda_j^{C}\right) - \left(1-\lambda_j^L\right)\right)\E(y)_j^2 - \left(\frac{\left(\lambda_{\max}^{C} \right)^{2}}{1 - \left(\lambda_{\max}^{C} \right)^{2}} \right) \frac{\sigma^2}{2}
\end{split}
\end{align}
where 
\begin{align}
    \E(y_t)_j^2 &= \E[(e_j^T y_t)^2]\\
    &= \E[e_j^T y_t y_t^T e_j]\\
    &= \E[e_j^T P^T(v_t - v_{t-1}) (v_t - v_{t-1})^T P e_j]\\
    &= e_j^T P^T \Sigma P e_j.
\end{align}
Therefore,
\begin{align}
\begin{split}
\E[\text{Cost}_{\FI{}}[1,T]] - \E[&\text{Cost}_{\text{\ai{}}}[1,T]]\\
\geq & \frac{T}{2} \sum_{j \not \in Z} \left(w\left(\lambda_j^A, \lambda_j^{C}\right) - \left(1-\lambda_j^L\right)\right)\left(e_j^T P^T \Sigma P e_j \right) - \left(\frac{\left(\lambda_{\max}^{C} \right)^{2}}{1 - \left(\lambda_{\max}^{C} \right)^{2}} \right) \frac{\sigma^2}{2}.
\end{split}
\end{align}
\end{proof}
\begin{proof}[Proof of Theorem \ref{thm:FI_lin_regret_short}]
Lastly, we show $(e_j^T P^T \Sigma P e_j) > 0$ for both cases of $\Sigma$. Let's consider the second case where $A$ and $\Sigma$ have same eigenvectors. We have $\Sigma = P D_\Sigma P^T$ giving us
\begin{align}
    \E(y)_j^2 &= e_j^T P^T \Sigma P e_j
    = D_\Sigma(j,j)
\end{align}
and
\begin{align}
\begin{split}
Regret_{\FI{}}[1,T] \geq \frac{T}{2} \sum_{j \not \in Z} \left(w\left(\lambda_j^A, \lambda_j^{C}\right) - \left(1-\lambda_j^L\right)\right)D_\Sigma(j,j) - \left(\frac{\left(\lambda_{\max}^{C} \right)^{2}}{1 - \left(\lambda_{\max}^{C} \right)^{2}} \right) \frac{\sigma^2}{2}.
\end{split}
\end{align}
Now, $(D-D_L)D_\Sigma = 0_{d \times d} \implies (C_{\robd{}} - C_L)\Sigma = 0_{d \times d}$. Therefore, the condition in the second case, that is, $(C- C_L)\Sigma \neq 0_{d \times d}$ implies, $\bar{D} = (D-D_L)D_\Sigma \neq 0_{d \times d}$, meaning, $\bar{D}(j,j) = 0$ if $j \in Z$ and $\bar{D}(j,j) = \left(\lambda_j^{C} - \lambda_j^L\right)D_\Sigma(j,j)$ if $j \not \in Z$. For $\bar{D} \neq 0$, there has to be some $j \not \in Z$ such that $D_\Sigma(j,j) > 0$ proving that \FI{} has $\Omega(T)$ regret.

For the case where $A$ and $\Sigma$ different eigenvectors, Lemma \ref{lemma:non-zero var} below shows that $\E(y)_j^2 > 0 \text{  } \forall \text{ } j$, giving $\Omega(T)$ regret.
\begin{lemma}\label{lemma:non-zero var}
Consider $A = P D_A P^T$, where $D_A$ is a diagonal matrix and a random vector $u_t \in \R^d$ with covariance matrix $\Sigma$. Suppose $\Sigma$ has modal matrix $Q \neq P$. Then $P^T u_t $ satisfies
\begin{align*}
    \E\left[e_j^T \left(P^T u_t u_t^T P\right) e_j \right] = e_j^T P^T \Sigma P e_j > 0 \text{ } \forall \text{ } j,t
\end{align*}
\end{lemma}

\textit{Proof.} We can write $\Sigma  = Q D_\Sigma Q^T$ where Q is the matrix having eigenvectors of $\Sigma$ as columns and $D_\Sigma$ is the diagonal matrix of eigenvalues of $\Sigma$. We will have $Q \neq P$ and
\begin{align}
    e_j^T P^T \Sigma P e_j &= e_j^T P^T Q D_\Sigma Q^T P e_j\\
    &= \sum_{k=1}^d D_\Sigma(k,k) (Q^T P e_j)_k^2
\end{align}
Now, notice that $Q^T P e_j$ is the decomposition of the $j^{th}$ eigenvector in $P$ into the components along the eigenbasis defined by Q. Since, $P$ and $Q$ are different eigenbases, $P e_j$ cannot be orthogonal to any eigenvectors in Q, that is any row of $Q^T$. Therefore each element of $Q^T P e_j$ is non-zero, giving us
\begin{align}
    (Q^T P e_j)_k^2 > 0 \text{ } \forall \text{ } k.
\end{align}
There has to be at least one positive diagonal element in $D_\Sigma$, otherwise $\Sigma = 0$ (as $\Sigma \succcurlyeq 0)$. This proves
\begin{align}
    e_j^T P^T \Sigma P e_j > 0 \text{ } \forall \text{ } j
\end{align}
\end{proof}

\subsection{Proof of Proposition \ref{prop:FI_cost_gen}}\label{proof:FI_cost_gen}

We calculate the cost of the following iterative algorithm, 
\begin{align}
    x_t = C x_{t-1} + (I-C) v_t
\end{align}
where $C \prec I$. The following claim holds.

\textbf{Claim:} The cost from round $T-t$ to $T$ is
\begin{align}
\begin{split}
\E[&\text{Cost}_{\FI{}}[T-t,T]] \\
=& \frac{1}{2} \E\left[ (x_{T-t-1} - v_{T-t})^T( AC^2 + (I-C)^2)\left( \sum_{i=0}^t C^{2i} \right)(x_{T-t-1} - v_{T-t}) \right]\\
&+ \sum_{s=1}^{t} \frac{1}{2}\E\left[(v_{T-s+1}-v_{T-s})^T(I-C^2)^{-1}(AC^2 + (I-C)^2)(I - C^{2s}) (v_{T-s+1}-v_{T-s}) \right]
\end{split}
\end{align}

We prove this claim by induction. First, we look at $t=0$ or cost at round $T$.
\begin{align}
\begin{split}
\E[\text{Cost}_{\FI{}}[T,T]] =& \frac{1}{2} \E\left[ (x_{T-1} - v_{T})^T C A C (x_{T-1} - v_{T}) \right]\\
&+ \frac{1}{2} \E\left[ (x_{T-1} - v_{T})^T (I-C) (I-C)(x_{T-1} - v_{T}) \right]
\end{split}\\
&= \frac{1}{2} \E\left[ (x_{T-1} - v_{T})^T( AC^2 + (I-C)^2) (x_{T-1} - v_{T}) \right]
\end{align}
Now, $A, C , (I-C)$ are all real symmetric matrices with the \textbf{same modal matrix $P$}. This means their \textbf{product is commutative}, which is used in the second inequality here.

Therefore, the claim holds for $t=0$ ($C^0 = I$ and sum from $1$ to $0$ is zero). Further, we take $0^0 = 1$. Assuming the claim holds for some $t \geq 0$, we have 
\begin{align}
\begin{split}
\E[&\text{Cost}_{\FI{}}[T-t,T]]\\
=& \frac{1}{2} \E\left[ (x_{T-t-1} - v_{T-t})^T( AC^2 + (I-C)^2)\left( \sum_{i=0}^t C^{2i} \right)(x_{T-t-1} - v_{T-t}) \right]\\
&+ \sum_{s=1}^{t} \frac{1}{2}\E\left[(v_{T-s+1}-v_{T-s})^T(I-C^2)^{-1}(AC^2 + (I-C)^2)(I - C^{2s}) (v_{T-s+1}-v_{T-s}) \right]
\end{split}\\
\begin{split}
=& \frac{1}{2} \E\left[ (x_{T-t-1} - v_{T-t-1})^T( AC^2 + (I-C)^2)\left( \sum_{i=0}^t C^{2i} \right)(x_{T-t-1} - v_{T-t-1}) \right]\\
&+ \E \left[ (x_{T-t-1} - v_{T-t-1})^T( AC^2 + (I-C)^2)\left( \sum_{i=0}^t C^{2i} \right) \E[v_{T-t} - v_{T-t-1} | \mathcal{F}_{T-t-1}]  \right]\\
&+ \frac{1}{2} \E\left[ (v_{T-t} - v_{T-t-1})^T( AC^2 + (I-C)^2)\left( \sum_{i=0}^t C^{2i} \right)(v_{T-t} - v_{T-t-1}) \right]\\
&+ \sum_{s=1}^{t} \frac{1}{2}\E\left[(v_{T-s+1}-v_{T-s})^T(I-C^2)^{-1}(AC^2 + (I-C)^2)(I - C^{2s}) (v_{T-s+1}-v_{T-s}) \right]
\end{split}\\
\begin{split}
=& \frac{1}{2} \E\left[ (x_{T-t-1} - v_{T-t-1})^T( AC^2 + (I-C)^2)\left( \sum_{i=0}^t C^{2i} \right)(x_{T-t-1} - v_{T-t-1}) \right]\\
&+ \frac{1}{2} \E\left[ (v_{T-t} - v_{T-t-1})^T( AC^2 + (I-C)^2)(I-C^2)^{-1}\left(I - C^{2(t+1)} \right)(v_{T-t} - v_{T-t-1}) \right]\\
&+ \sum_{s=1}^{t} \frac{1}{2}\E\left[(v_{T-s+1}-v_{T-s})^T(I-C^2)^{-1}(AC^2 + (I-C)^2)(I - C^{2s}) (v_{T-s+1}-v_{T-s}) \right]
\end{split}\\
\begin{split}
=& \frac{1}{2} \E\left[ (x_{T-t-1} - v_{T-t-1})^T( AC^2 + (I-C)^2)\left( \sum_{i=0}^t C^{2i} \right)(x_{T-t-1} - v_{T-t-1}) \right]\\
&+ \sum_{s=1}^{t+1} \frac{1}{2}\E\left[(v_{T-s+1}-v_{T-s})^T(I-C^2)^{-1}(AC^2 + (I-C)^2)(I - C^{2s}) (v_{T-s+1}-v_{T-s}) \right]
\end{split}
\end{align}
The cost from round $T-t-1$ to $T$, therefore, is
\begin{align}
\begin{split}
\E[&\text{Cost}_{\FI{}}[T-t-1,T]]\\
=& \frac{1}{2} \E\left[ (x_{T-t-1} - v_{T-t-1})^T A (x_{T-t-1} - v_{T-t-1})\right] + \frac{1}{2} \E\left[ \|x_{T-t-1} - x_{T-t-2}\|_2^2 \right]\\
&+ \E[C_{\FI{}}[T-t,T]]
\end{split}\\
\begin{split}
=& \frac{1}{2} \E\left[ (x_{T-t-1} - v_{T-t-1})^T \left(( AC^2 + (I-C)^2)\left( \sum_{i=0}^t C^{2i} \right) + A \right) (x_{T-t-1} - v_{T-t-1})\right]\\
&+ \frac{1}{2} \E\left[ \|x_{T-t-1} - x_{T-t-2}\|_2^2 \right]\\
&+ \sum_{s=1}^{t+1} \frac{1}{2}\E\left[(v_{T-s+1}-v_{T-s})^T(I-C^2)^{-1}(AC^2 + (I-C)^2)(I - C^{2s}) (v_{T-s+1}-v_{T-s}) \right]
\end{split}\\
\begin{split}
=& \frac{1}{2} \E\left[ (x_{T-t-2} - v_{T-t-1})^T C\left(( AC^2 + (I-C)^2)\left( \sum_{i=0}^t C^{2i} \right) + A \right)C (x_{T-t-2} - v_{T-t-1})\right]\\
&+ \frac{1}{2} \E\left[ (x_{T-t-2} - v_{T-t-1})^T (I-C)^2 (x_{T-t-2} - v_{T-t-1})\right]\\
&+ \sum_{s=1}^{t+1} \frac{1}{2}\E\left[(v_{T-s+1}-v_{T-s})^T(I-C^2)^{-1}(AC^2 + (I-C)^2)(I - C^{2s}) (v_{T-s+1}-v_{T-s}) \right]
\end{split}\\
\begin{split}
=& \frac{1}{2} \E\Bigg[ (x_{T-t-2} - v_{T-t-1})^T \left(( AC^2 + (I-C)^2)\left( \sum_{i=1}^{t+1} C^{2i} \right) + AC^2 + (I-C)^2 \right) \ldots\\
&(x_{T-t-2} - v_{T-t-1}) \Bigg]\\
&+ \sum_{s=1}^{t+1} \frac{1}{2}\E\left[(v_{T-s+1}-v_{T-s})^T(I-C^2)^{-1}(AC^2 + (I-C)^2)(I - C^{2s}) (v_{T-s+1}-v_{T-s}) \right]
\end{split}\\
\begin{split}
=& \frac{1}{2} \E\left[ (x_{T-t-2} - v_{T-t-1})^T ( AC^2 + (I-C)^2)\left( \sum_{i=0}^{t+1} C^{2i} \right) (x_{T-t-2} - v_{T-t-1}) \right]\\
&+ \sum_{s=1}^{t+1} \frac{1}{2}\E\left[(v_{T-s+1}-v_{T-s})^T(I-C^2)^{-1}(AC^2 + (I-C)^2)(I - C^{2s}) (v_{T-s+1}-v_{T-s}) \right]
\end{split}
\end{align}
Now, to complete the proof, the total cost is
\begin{align}
\begin{split}
\E[&\text{Cost}_{\FI{}}[1,T]]\\
=& \frac{1}{2} \E\left[ (x_{0} - v_{1})^T(I-C^2)^{-1}(AC^2 + (I-C)^2)(I - C^{2T})(x_{0} - v_{1}) \right]\\
&+ \sum_{s=1}^{T-1} \frac{1}{2}\E\left[(v_{T-s+1}-v_{T-s})^T(I-C^2)^{-1}(AC^2 + (I-C)^2)(I - C^{2s}) (v_{T-s+1}-v_{T-s}) \right]
\end{split}\\
\begin{split}
=& \sum_{s=1}^{T} \frac{1}{2}\E\left[(v_{T-s+1}-v_{T-s})^T(I-C^2)^{-1}(AC^2 + (I-C)^2) (v_{T-s+1}-v_{T-s}) \right]\\
&-\sum_{s=1}^{T} \frac{1}{2}\E\left[(v_{T-s+1}-v_{T-s})^T C^{2s} (v_{T-s+1}-v_{T-s}) \right]
\end{split}\\
\begin{split}
\geq & \sum_{s=1}^{T} \frac{1}{2}\E\left[(v_{T-s+1}-v_{T-s})^T(I-C^2)^{-1}(AC^2 + (I-C)^2) (v_{T-s+1}-v_{T-s}) \right]\\
&-\sum_{s=1}^{T} \frac{1}{2}\E\left[(v_{T-s+1}-v_{T-s})^T \left(\lambda_{\max}^C \right)^{2s} I (v_{T-s+1}-v_{T-s}) \right]
\end{split}\\
\begin{split}
\geq & \sum_{s=1}^{T} \frac{1}{2}\E\left[(v_{T-s+1}-v_{T-s})^T(I-C^2)^{-1}(AC^2 + (I-C)^2) (v_{T-s+1}-v_{T-s}) \right]\\
&- \left(\frac{\left(\lambda_{\max}^C \right)^{2}}{1 - \left(\lambda_{\max}^C \right)^{2}} \right) \frac{\sigma^2}{2}
\end{split}\\
\begin{split}
\geq & \sum_{t=1}^{T} \frac{1}{2}\E\left[(v_{t}-v_{t-1})^T\left(I-C^2\right)^{-1}(AC^2 + \left(I-C\right)^2) (v_{t}-v_{t-1}) \right]\\
&- \left(\frac{\left(\lambda_{\max}^{C} \right)^{2}}{1 - \left(\lambda_{\max}^{C} \right)^{2}} \right) \frac{\sigma^2}{2}
\end{split}\label{eqn:FI_cost_ineq}
\end{align}
\[
\pushQED{\qed} 
\qedhere
\popQED
\]  

\subsection{Proof of Theorem \ref{thm:NOFI_AI_perf} (i)}\label{proof:NOFI_AI_perf_stoch}
The analysis of the upper bound on the regret of \LAItxt{} involves two key steps. The first is an upper bound on the cost of \LAItxt{}, as stated in the following proposition.
\begin{proposition}\label{prop:nofi_dp_cost}
The total cost of \LAItxt{} for martingale minimizers \eqref{eqn:minimizer_martingale} is upper bounded as
\begin{align}
\begin{split}
\E[\text{Cost}_{\LAImath{}}[1,T]] &\leq \frac{1}{2}\sum_{t=1}^{T}  \E\left[(v_{t} - v_{t-1})^T (I-\Tilde{C}_{t}) (v_{t} - v_{t-1})\right].
\end{split}
\end{align}
\end{proposition}
The proof of this proposition uses an induction argument by cleverly employing the recursion followed by $\{\Tilde{C}_t\}_t$ and is proved right after the theorem proof in Appendix \ref{proof:nofi_dp_cost}. The second step uses the following key property exhibited by eigenvalues of $\{\Tilde{C}_t\}_t$ of \LAItxt{} and $\{C_t\}_t$ of \ai{}.
\begin{lemma}\label{lemma:nofi_dp_eigenvalues}
The $i^{th}$ eigenvalue of \LAItxt{}'s $\Tilde{C}_t$, that is $\Tilde{\lambda}_i^t$, and the $i^{th}$ eigenvalue of \ai{}'s $C_t$, that is, $\lambda_i^t$, are related in the following manner
\begin{align*}
\left(\lambda_i^t - \Tilde{\lambda}_i^t \right) &\leq \frac{\lambda_i^A}{2}\left( \left(1 + \frac{4}{\lambda_i^A}\right)^{\frac{\gamma}{2}} - 1 \right) \left(\lambda_i^T \right)^{2(T-t+1)}
\end{align*}
\end{lemma}
The proof of this lemma uses the recursion followed by both $\{\Tilde{C}_t\}_t$ and $\{\Tilde{C}_t\}_t$ and is presented in Appendix \ref{proof:nofi_dp_eigenvalues}. We now proceed to upper bound the regret of \LAItxt{} algorithm.

\begin{proof}[Proof of Theorem~\ref{thm:NOFI_AI_perf}]
We write each $C_t$ and $\Tilde{C}_t$ as $PD_t P^T$ and $P \Tilde{D}_t P^T$ where $D_t$ and $\Tilde{D}_t$ are the eigenvalue diagonal matrices of $C_t$ and $\Tilde{C}_t$ respectively. Therefore, regret of \LAItxt{} is \textbf{upper-bounded} as
\begin{align}
&\E[Cost_{\LAImath{}}[1,T]] - \E[\text{Cost}_{\text{\ai{}}}[1,T]] \leq \frac{1}{2}\sum_{t=1}^{T}  \E\left[(v_{t} - v_{t-1})^T (C_t-\Tilde{C}_{t}) (v_{t} - v_{t-1})\right]\\
&= \frac{1}{2}\sum_{t=1}^{T}  \E\left[(v_{t} - v_{t-1})^T P(D_t-\Tilde{D}_{t}) P^T(v_{t} - v_{t-1})\right]
\leq \frac{1}{2}\sum_{t=1}^{T} \sum_{i=1}^d\left(\lambda_i^t - \Tilde{\lambda}_i^t \right) \E(y_t)_i^2\\
&= \frac{1}{2}\sum_{i=1}^d \E(y)_i^2 \sum_{t=1}^{T} \left(\lambda_i^t - \Tilde{\lambda}_i^t \right) 
\leq  \frac{1}{2}\sum_{i=1}^d \E(y)_i^2 \sum_{t=1}^{T} \frac{\lambda_i^A}{2}\left( \left(1 + \frac{4}{\lambda_i^A}\right)^{\frac{\gamma}{2}} - 1 \right) \left(\lambda_i^T \right)^{2(T-t+1)} \\
&\leq \frac{1}{2}\sum_{i=1}^d \E(y)_i^2 \frac{\lambda_i^A}{2}\left( \left(1 + \frac{4}{\lambda_i^A}\right)^{\frac{\gamma}{2}} - 1 \right) \sum_{t=1}^{\infty}  \left(\lambda_i^T \right)^{2t} 
=  \frac{1}{2}\sum_{i=1}^d \E(y)_i^2 \frac{\lambda_i^A}{2}\left( \left(1 + \frac{4}{\lambda_i^A}\right)^{\frac{\gamma}{2}} - 1 \right) \frac{\left(\lambda_i^T \right)^2}{1-\left(\lambda_i^T \right)^2}\\
&=  \frac{1}{2}\sum_{i=1}^d \E(y)_i^2 \frac{\left( \left(1 + \frac{4}{\lambda_i^A}\right)^{\frac{\gamma}{2}} - 1 \right)}{2\left(\lambda_i^A + 2 \right)}\\
\end{align}
where $y_t = P^T(v_t - v_{t-1})$. Note that $\E(y_t)_i^2 = \E(y)_i^2$ is time-invariant as all $(v_t - v_{t-1})$ have same covariance $\Sigma$. The function $\left( \left(1 + \frac{4}{q}\right)^{\frac{\gamma}{2}} - 1 \right)/(2\left(q + 2 \right))$ is always decreasing in $q$ for $\gamma \in [0,1]$. Therefore, 
\begin{align}
\text{Regret}_{\LAImath{}}[1,T] \leq \frac{1}{4} \frac{\left( \left(1 + \frac{4}{\lambda_{\min}^A}\right)^{\frac{\gamma}{2}} - 1 \right)}{\left(\lambda_{\min}^A + 2 \right)} \sum_{i=1}^d \E(y)_i^2  = \frac{\left( \left(1 + \frac{4}{\lambda_{\min}^A}\right)^{\frac{\gamma}{2}} - 1 \right)}{\left(\lambda_{\min}^A + 2 \right)} \frac{\sigma^2}{4} 
\end{align}
where the last step uses $\sum_{i=1}^d \E(y)_i^2 = \sigma^2 = tr(\Sigma)$ which is stated and proved as Lemma \ref{lemma:var_split_and_combine}.
\end{proof}

\subsection{Proof of Proposition \ref{prop:nofi_dp_cost}}\label{proof:nofi_dp_cost}
\textbf{Claim:} The cost from round $T-t$ to $T$ for martingale minimizers is upper bound by
\begin{align}
\begin{split}
\E[\text{Cost}_{\LAImath{}}[T-t,T]] \leq& \frac{1}{2}\E\left[(x_{T-t-1} - v_{T-t})^T (I-\Tilde{C}_{T-t})(x_{T-t-1} - v_{T-t})\right]\\
&+ \frac{1}{2}\sum_{s=0}^{t-1}  \E\left[(v_{T-s} - v_{T-s-1})^T (I-\Tilde{C}_{T-s}) (v_{T-s} - v_{T-s-1})\right]
\end{split}
\end{align}
For $t=0$, it is the expected cost of the final round, which is
\begin{align}
    \E\left[\text{Cost}_{\LAImath{}}[T,T]\right] =& \frac{1}{2} \E\left[(x_{T-1} - v_T)^T \left(A(\Tilde{C}_T)^2 + (I-\Tilde{C}_T)^2 \right) (x_{T-1} - v_T)\right]\\
    \begin{split}
    =& \frac{1}{2} \E\left[(x_{T-1} - v_T)^T (I - \Tilde{C}_T) (x_{T-1} - v_T)\right]\\
    &+ \frac{1}{2} \E\left[(x_{T-1} - v_T)^T \left(A(\Tilde{C}_T)^2 + (I-\Tilde{C}_T)^2 - (I - \Tilde{C}_T) \right)  (x_{T-1} - v_T)\right]
    \end{split}\\
    \begin{split}
    =& \frac{1}{2} \E\left[(x_{T-1} - v_T)^T (I - \Tilde{C}_T) (x_{T-1} - v_T)\right]\\
    &+ \frac{1}{2} \E\left[(x_{T-1} - v_T)^T \left(A(\Tilde{C}_T)^2 - \Tilde{C}_T(I-\Tilde{C}_T) \right)  (x_{T-1} - v_T)\right]
    \end{split}\\
    \begin{split}
    =& \frac{1}{2} \E\left[(x_{T-1} - v_T)^T (I - \Tilde{C}_T) (x_{T-1} - v_T)\right]\\
    &+ \frac{1}{2} \E\left[(x_{T-1} - v_T)^T \left( \Tilde{C}_T(A+I) (\Tilde{C}_T - (A+I)^{-1}) \right)  (x_{T-1} - v_T)\right]
    \end{split}
\end{align}
Now, observe two things. First that $\Tilde{\lambda}_i^T < \frac{1}{\lambda_i^A + 1}$ and second that $\Tilde{C}_T$ has the same set of eigenvectors as $A$. Therefore, $\Tilde{C}_T(A+I) (\Tilde{C}_T - (A+I)^{-1})$ will be negative definite as $\Tilde{C}_T \prec (A+I)^{-1}$, $\Tilde{C}_T \succ 0$ and $(A+I) \succ I$. Therefore,
\begin{align}
    \E[\text{Cost}_{\LAImath{}}[T,T]] \leq \frac{1}{2}\E\left[ (x_{T-1} - v_T)^T (I - \Tilde{C}_T) (x_{T-1} - v_T)\right]
\end{align}
Therefore, the claim holds for $t=0$ as the sum from $-1$ to $0$ is zero. Assuming the claim holds for some $t \geq 0$,
\begin{align}
\begin{split}
    \E[\text{Cost}_{\LAImath{}}[T-t,T]] \leq& \frac{1}{2}\E\left[(x_{T-t-1} - v_{T-t})^T (I-\Tilde{C}_{T-t})(x_{T-t-1} - v_{T-t})\right]\\
    &+ \frac{1}{2}\sum_{s=0}^{t-1}  \E\left[(v_{T-s} - v_{T-s-1})^T (1-\Tilde{C}_{T-s}) (v_{T-s} - v_{T-s-1})\right]
\end{split}\\
\begin{split}
    \leq& \frac{1}{2}\E\left[(x_{T-t-1} - v_{T-t-1})^T (I-\Tilde{C}_{T-t})(x_{T-t-1} - v_{T-t-1})\right]\\
    &+ \E\left[ (x_{T-t-1} - v_{T-t-1})^T (I-\Tilde{C}_{T-t}) (v_{T-t-1} - v_{T-t})\right]\\
    &+ \frac{1}{2}\E\left[ (v_{T-t-1} - v_{T-t})^T (I-\Tilde{C}_{T-t}) (v_{T-t-1} - v_{T-t})\right]\\
    &+ \frac{1}{2}\sum_{s=0}^{t-1} \E\left[(v_{T-s} - v_{T-s-1})^T (I-\Tilde{C}_{T-s}) (v_{T-s} - v_{T-s-1})\right]
\end{split}\\
\begin{split}
    \leq& \frac{1}{2}\E\left[(x_{T-t-1} - v_{T-t-1})^T (I-\Tilde{C}_{T-t})(x_{T-t-1} - v_{T-t-1})\right]\\
    &+ \E\left[ (x_{T-t-1} - v_{T-t-1})^T (I-\Tilde{C}_{T-t}) \E\left[v_{T-t-1} - v_{T-t} | \mathcal{F}_{T-t-1}\right]\right]\\
    &+ \frac{1}{2}\sum_{s=0}^{t} \E\left[(v_{T-s} - v_{T-s-1})^T (I-\Tilde{C}_{T-s}) (v_{T-s} - v_{T-s-1})\right]
\end{split}\\
\begin{split}
    \leq& \frac{1}{2}\E\left[(x_{T-t-1} - v_{T-t-1})^T (I-\Tilde{C}_{T-t})(x_{T-t-1} - v_{T-t-1})\right]\\
    &+ \frac{1}{2}\sum_{s=0}^{t} \E\left[(v_{T-s} - v_{T-s-1})^T (I-\Tilde{C}_{T-s}) (v_{T-s} - v_{T-s-1})\right]
\end{split}
\end{align}
Therefore, the cost from $T-t-1$ to $T$ is
\begin{align}
\begin{split}
\E[Cos&t_{\LAImath{}}[T-t-1,T]]\\
\leq& \frac{1}{2}\E\left[(x_{T-t-1} - v_{T-t-1})^T A(x_{T-t-1} - v_{T-t-1})\right]\\
&+ \frac{1}{2} \E\|x_{T-t-1} - x_{T-t-2}\|_2^2 +  \E[\Tilde{C}_{\LAImath{}}[T-t,T]]
\end{split}\\
\begin{split}
\leq& \frac{1}{2}\E\left[(x_{T-t-1} - v_{T-t-1})^T (A+I-\Tilde{C}_{T-t})(x_{T-t-1} - v_{T-t-1})\right]\\
&+ \frac{1}{2} \E\left[(x_{T-t-1} - x_{T-t-2})^T (x_{T-t-1} - x_{T-t-2})\right]\\
&+ \frac{1}{2}\sum_{s=0}^{t} \E\left[(v_{T-s} - v_{T-s-1})^T (I-\Tilde{C}_{T-s}) (v_{T-s} - v_{T-s-1})\right] 
\end{split}\\
\begin{split}
\leq& \frac{1}{2}\E[(x_{T-t-2} - v_{T-t-1})^T \Tilde{C}_{T-t-1}(\Tilde{C}_{T-t-1}^{-1} - I)\Tilde{C}_{T-t-1}(x_{T-t-2} - v_{T-t-1})]\\
&+ \frac{1}{2} \E\left[(v_{T-t-1} - x_{T-t-2})^T (I - \Tilde{C}_{T-t-1})^2(v_{T-t-1} - x_{T-t-2})\right]\\
&+ \frac{1}{2}\sum_{s=0}^{t} \E\left[(v_{T-s} - v_{T-s-1})^T (I-\Tilde{C}_{T-s}) (v_{T-s} - v_{T-s-1})\right] 
\end{split}\\
\begin{split}
\leq& \frac{1}{2}\E[(x_{T-t-2} - v_{T-t-1})^T \left( \Tilde{C}_{T-t-1}(I - \Tilde{C}_{T-t-1}) + (I - \Tilde{C}_{T-t-1})^2 \right) (x_{T-t-2} - v_{T-t-1})]\\
&+ \frac{1}{2}\sum_{s=0}^{t} \E\left[(v_{T-s} - v_{T-s-1})^T (I-\Tilde{C}_{T-s}) (v_{T-s} - v_{T-s-1})\right] 
\end{split}\\
\begin{split}
\leq& \frac{1}{2}\E[(x_{T-t-2} - v_{T-t-1})^T (\Tilde{C}_{T-t-1} + I - \Tilde{C}_{T-t-1})(I - \Tilde{C}_{T-t-1}) (x_{T-t-2} - v_{T-t-1})]\\
&+ \frac{1}{2}\sum_{s=0}^{t} \E\left[(v_{T-s} - v_{T-s-1})^T (I-\Tilde{C}_{T-s}) (v_{T-s} - v_{T-s-1})\right] 
\end{split}\\
\begin{split}
\leq& \frac{1}{2}\E[(x_{T-t-2} - v_{T-t-1})^T (I - \Tilde{C}_{T-t-1}) (x_{T-t-2} - v_{T-t-1})]\\
&+ \frac{1}{2}\sum_{s=0}^{t} \E\left[(v_{T-s} - v_{T-s-1})^T (I-\Tilde{C}_{T-s}) (v_{T-s} - v_{T-s-1})\right] 
\end{split}
\end{align}
Now that our claim is proven, we can get the total cost of the algorithm by putting $t = T-1$
\begin{align}
\begin{split}
\E[\text{Cost}_{\LAImath{}}[1,T]] \leq& \frac{1}{2}\E\left[(x_{0} - v_{1})^T (I-\Tilde{C}_{1})(x_{0} - v_{1})\right]\\
&+ \frac{1}{2}\sum_{s=0}^{T-2}  \E\left[(v_{T-s} - v_{T-s-1})^T (I-\Tilde{C}_{T-s}) (v_{T-s} - v_{T-s-1})\right]
\end{split}\\
\begin{split}
=& \frac{1}{2}\E\left[(x_{0} - v_{1})^T (I-\Tilde{C}_{1})(x_{0} - v_{1})\right]\\
&+ \frac{1}{2}\sum_{t=2}^{T}  \E\left[(v_{t} - v_{t-1})^T (I-\Tilde{C}_{t}) (v_{t} - v_{t-1})\right]
\end{split}
\end{align}
\[
\pushQED{\qed} 
\qedhere
\popQED
\]    

\subsection{Proof of Theorem \ref{thm:adv_analysis_framework}}\label{proof:adv_analysis_framework}
The proof technique used here is inspired from \robd{}'s competitive analysis in \cite{GoelLinWierman19}. The key aspect of our proof is the potential function that depends on $g_t(\cdot)$ of the online algorithm \alg{} in question. To recap, the online algorithm \alg{} is
\begin{align}
    x_t = \argmin_x f_t(x) + D_{h}(x||x_{t-1}) + D_{g_t}(x||v_t)
\end{align}
where $f_t$ is $m$-strongly convex and $h(\cdot)$ is $\alpha$-strongly convex and $\beta$-smooth. Recall that $g_t(\cdot)$ is $\alpha'-$strongly convex and $\beta'-$smooth. This means
\begin{align}
    \nabla f_t(x_t) + (\nabla h(x_t) - \nabla h(x_{t-1})) +  (\nabla g_t(x_t) - \nabla g_t(v_t)) = 0.
\end{align}
We denote the hindsight optimal action sequence as $\left\{ x_t^*\right\}_t$. Using strong convexity of $f_t$
\begin{align}
f_t(x_t^*) \geq& f_t(x_t) + \langle \nabla f_t(x_t), x_t^* - x_t\rangle + \frac{m}{2}\norm{x_t^* - x_t}^2\\
=& f_t(x_t) - \langle \nabla h(x_{t-1}) - \nabla h(x_t), x_t - x_t^* \rangle\\
    &-  \langle \nabla g_t(v_t) - \nabla g_t(x_t), x_t - x_t^* \rangle + \frac{m}{2}\norm{x_t^* - x_t}^2.\label{eqn9}
\end{align}

\begin{lemma}[Generalized Triangle Inequality]\label{lemma:breg_div_3_pts}
For any three points $x,y,x \in \R^n$, the following holds for Bregman divergence $D_h(\cdot||\cdot)$ defined using a strongly convex function $h(\cdot)$
\begin{align*}
    \langle \nabla h(y) - \nabla h(z), z - x \rangle = D_h(x || y) - D_h(x || z) - D_h(z || y)
\end{align*}
\end{lemma}

Using the above lemma for $D_h \left(\cdot||\cdot \right)$ and $D_{g_t} \left(\cdot||\cdot \right)$, we have
\begin{align}
    \langle \nabla h(x_{t-1}) - \nabla h(x_t), x_t - x_t^* \rangle = D_h(x_t^* || x_{t-1}) - D_h(x_t^* || x_t) - D_h(x_t || x_{t-1})
\end{align}
and
\begin{align}
\langle \nabla g_t(v_t) - \nabla g_t(x_t), x_t - x_t^* \rangle = D_{g_t}(x_t^* || v_t) - D_{g_t}(x_t^* || x_t) - D_{g_t}(x_t || v_t).
\end{align}
Substituting the two above identities into inequality \eqref{eqn9}, we get
\begin{align}
    &f_t(x_t) + D_h(x_t || x_{t-1}) + (D_h(x_t^* || x_t) + D_{g_t}(x_t^* || x_t)) + \frac{m}{2}\norm{x_t^* - x_t}^2 + D_{g_t}(x_t || v_t)\\
    &\leq f_t(x_t^*) + D_{g_t}(x_t^* || v_t) + D_h(x_t^* || x_{t-1})
\end{align}
It follows that 
\begin{align}\label{eqn1}
\begin{split}
&f_t(x_t) + D_h(x_t || x_{t-1}) + (D_h(x_t^* || x_t) + D_{g_t}(x_t^* || x_t)) + \frac{m}{2}\norm{x_t^* - x_t}^2\\
    &\leq f_t(x_t^*) + D_{g_t}(x_t^* || v_t) + D_h(x_t^* || x_{t-1})
\end{split}  
\end{align}
Now, define the potential function $\phi(x_t,x_t^*) = (D_h(x_t^* || x_t) + D_{g_t}(x_t^* || x_t)) + \frac{m}{2}\norm{x_t^* - x_t}^2$ and the potential difference as $\Delta \phi = \phi(x_t,x_t^*) - \phi(x_{t-1},x_{t-1}^*)$. Subtracting $\phi(x_{t-1},x_{t-1}^*)$ from both sides of the above inequality, we get
\begin{align}\label{eqn6}
\begin{split}
&f_t(x_t) + D_h(x_t || x_{t-1}) + \Delta \phi \\
    &\leq \underbrace{(f_t(x_t^*) + D_{g_t}(x_t^* || v_t))}_{X_1} + \underbrace{ D_h(x_t^* || x_{t-1}) - (D_h(x_{t-1}^* || x_{t-1}) + D_{g_{t-1}}(x_{t-1}^* || x_{t-1})) - \frac{m}{2}\norm{x_{t-1}^* - x_{t-1}}^2}_{X_2}
\end{split}
\end{align}
We now analyze $X_2$
\begin{align}
D_h(&x_t^* || x_{t-1}) - \left(D_h(x_{t-1}^* || x_{t-1}) + D_{g_{t-1}}(x_{t-1}^* || x_{t-1}) \right) - \frac{m}{2}\norm{x_{t-1}^* - x_{t-1}}^2\label{eqn2}\\
=&\text{ }(D_h(x_t^* || x_{t-1}) - D_h(x_{t-1}^* || x_{t-1})) - D_{g_{t-1}}(x_{t-1}^* || x_{t-1}) - \frac{m}{2}\norm{x_{t-1}^* - x_{t-1}}^2\\
\begin{split}
=& \text{ }D_h(x_t^* || x_{t-1}^*) + \langle \nabla h(x_{t-1}) - \nabla h(x_{t-1}^*), x_{t-1}^* - x_t^* \rangle - D_{g_{t-1}}(x_{t-1}^* || x_{t-1}) - \frac{m}{2}\norm{x_{t-1}^* - x_{t-1}}^2
\end{split}\\
\begin{split}
\leq& \text{ }D_h(x_t^* || x_{t-1}^*) + \|\nabla h(x_{t-1}) - \nabla h(x_{t-1}^*)\|_2 \| x_{t-1}^* - x_t^* \|_2 - D_{g_{t-1}}(x_{t-1}^* || x_{t-1}) - \frac{m}{2}\norm{x_{t-1}^* - x_{t-1}}^2
\end{split}\\
\begin{split}
\leq& \text{ }D_h(x_t^* || x_{t-1}^*) + \beta\|x_{t-1} - x_{t-1}^*\|_2 \| x_{t-1}^* - x_t^* \|_2 - D_{g_{t-1}}(x_{t-1}^* || x_{t-1}) - \frac{m}{2}\norm{x_{t-1}^* - x_{t-1}}^2
\end{split}\\
\begin{split}
\leq& \text{ }D_h(x_t^* || x_{t-1}^*) + \frac{1}{2} \left(\frac{\beta^2}{\alpha'_{t-1} + m} \right) \| x_{t-1}^* - x_t^* \|_2^2 + \left(\frac{\alpha' + m}{2}\right) \|x_{t-1} - x_{t-1}^*\|_2^2\\
&- D_{g_{t-1}}(x_{t-1}^* || x_{t-1}) - \frac{m}{2}\norm{x_{t-1}^* - x_{t-1}}^2
\end{split}\label{eqn3}\\
\begin{split}
=&\text{ } D_h(x_t^* || x_{t-1}^*) + \frac{1}{2} \left(\frac{\beta^2}{\alpha'_{t-1} + m} \right) \| x_{t-1}^* - x_t^* \|_2^2 + \left( \frac{\alpha'_{t-1}}{2} \|x_{t-1} - x_{t-1}^*\|_2^2  - D_{g_{t-1}}(x_{t-1}^* || x_{t-1}) \right)
\end{split}\label{eqn4}\\
\begin{split}
\leq& \text{ }D_h(x_t^* || x_{t-1}^*) + \frac{1}{2} \left(\frac{\beta^2}{\alpha'_{t-1} + m} \right) \| x_{t-1}^* - x_t^* \|_2^2
\end{split} \leq \left( 1 + \frac{\beta^2/ \alpha}{\alpha'_{t-1} + m} \right) D_h(x_t^* || x_{t-1}^*) \label{eqn8}
\end{align}
Therefore,
\begin{align}
f_t(x_t) + D_h(x_t || x_{t-1}) + \Delta \phi &\leq f_t(x_t^*) + D_{g_{t}}(x_t^* || v_t) + \left( 1 + \frac{\beta^2/ \alpha}{\alpha'_{t-1} + m} \right) D_h(x_t^* || x_{t-1}^*)\label{eqn5}\\
& \leq f_t(x_t^*) + \frac{\beta'_t}{2}\|x_t^* - v_t\|^2 + \left( 1 + \frac{\beta^2/ \alpha}{\alpha'_{t-1} + m} \right) D_h(x_t^* || x_{t-1}^*)\\
&= \left( 1 + \frac{\beta'_t}{m} \right)f_t(x_t^*) + \left( 1 + \frac{\beta^2/ \alpha}{\alpha'_{t-1} + m} \right) D_h(x_t^* || x_{t-1}^*)\\
&\leq \left( 1 + \frac{\max_t \beta'_t}{m} \right)f_t(x_t^*) + \left( 1 + \frac{\beta^2/ \alpha}{\min_t \alpha'_{t-1} + m} \right) D_h(x_t^* || x_{t-1}^*)
\end{align}
\[
\pushQED{\qed} 
\qedhere
\popQED
\]  

\subsection{Proofs of Theorem \ref{thm:NOFI_AI_perf} (ii) and Corollary \ref{corr:LAI_best_of_both}}\label{proof:NOFI_AI_CR_proofs}
To bound the competitive ratio of \LAItxt{}, we use Theorem \ref{thm:adv_analysis_framework} by writing \LAItxt{} in the form: 
\begin{align}
    A(x_t - v_t) + (x_t - x_{t-1}) + (\Tilde{C}_t^{-1} - I - A)(x_t - v_t) = 0
\end{align}
or in the form of an optimization problem as
\begin{align}
    x_t = \argmin_x \frac{1}{2} (x_t - v_t)^T A(x_t - v_t) + \frac{1}{2}\|x_{t} - x_{t-1}\|_2^2 + \frac{1}{2} (x_t - v_t)^T (\Tilde{C}_t^{-1} - I - A)(x_t - v_t).
\end{align}
Therefore, $g_t(x) = \frac{1}{2}x^T (\Tilde{C}_t^{-1} - 
I - A) x$. Now, the eigenvalues satisfy the recursion
$\frac{1}{\Tilde{\lambda}_i^t} = 2 + \lambda_i^A - \Tilde{\lambda}_i^{t+1}$
as result of $\Tilde{C}_t^{-1} = 2I + A - \Tilde{C}_{t+1}$ and $\Tilde{C}_t$ having same eigenvectors as $A$. The above property, along with $\Tilde{\lambda}_i^T \geq \lambda_i^L$ (as $\gamma \in [0,1]$), gives $\Tilde{\lambda}_i^t \leq  \Tilde{\lambda}_i^{t+1}$ $\forall$ $i$ and
$C_L \preccurlyeq \Tilde{C}_1 \preccurlyeq \ldots \preccurlyeq \Tilde{C}_T$.
The proof of this fact is straightforward and is done in Corollary \ref{corr:Ct_add_props}. Consequently,
\begin{align}
    C_L^{-1} - (I+A) \succcurlyeq \Tilde{C}_1^{-1} - (I+A) \succcurlyeq \ldots  \succcurlyeq \Tilde{C}_T^{-1} - (I+A)
\end{align}
or
\begin{align}
    \frac{1}{\lambda_i^L} - 1 - \lambda_i^A \geq \frac{1}{\Tilde{\lambda}_i^1} - 1 - \lambda_i^A \geq \ldots \geq \frac{1}{\Tilde{\lambda}_i^T} - 1 - \lambda_i^A.
\end{align}
This means
\begin{align}
    \max_t \beta'_t &= \max_i \left\{ \frac{1}{\lambda_i^L} - 1 - \lambda_i^A \right\} = \frac{\sqrt{\left(\lambda_{\max}^A\right)^2 + 4 \lambda_{\max}^A} - \lambda_{\max}^A}{2} = \frac{\lambda_{\min}^A}{2}\left(\sqrt{\kappa(A)^2 +  \frac{4\kappa(A)}{\lambda_{\min}^A}} - \kappa(A) \right)\\
    \min_{t} \alpha'_{t-1} &=  \min_i \left\{ \frac{1}{\Tilde{\lambda}_i^T} - 1 - \lambda_i^A \right\} = \min_i \left\{\frac{\lambda_i^A}{2}\left( \left(1 + \frac{4}{\lambda_i^A}\right)^{\frac{\gamma}{2}} - 1 \right) \right\}
\end{align}
Since,  $\frac{y}{2}\left( \left(1 + \frac{4}{y}\right)^{\frac{\gamma}{2}} - 1 \right)$ is a non-decreasing function in $y \geq 0$ and any $\gamma \in [0,1]$, we have
\begin{align}
    \min_{t} \alpha'_{t-1} &= \frac{\lambda_{\min}^A}{2}\left( \left(1 + \frac{4}{\lambda_{\min}^A}\right)^{\frac{\gamma}{2}} - 1 \right)
\end{align}
Using Theorem \ref{thm:adv_analysis_framework}, 
\begin{align}
    \text{CR}_{\LAImath{}} = \max \left\{ 1 + \frac{1}{2}\left(\sqrt{\kappa(A)^2 +  \frac{4\kappa(A)}{\lambda_{\min}^A}} - \kappa(A) \right) , 1 + \frac{1}{\frac{\lambda_{\min}^A}{2}\left( \left(1 + \frac{4}{\lambda_{\min}^A}\right)^{\frac{\gamma}{2}} + 1 \right)} \right\}
\end{align}
\[
\pushQED{\qed} 
\qedhere
\popQED
\] 

\begin{corollary}\label{corr:NOFI_AI_cr_order}
In the adversarial setting where $\lambda_{\max}^A \ll 1$ , \LAItxt{} has a competitive ratio of 
\begin{equation*}
    \text{CR}_{\LAImath{}} \leq 1 + \max \left\{ \sqrt{\frac{\kappa(A)}{\lambda_{\min}^A}} , \frac{2^{1-\gamma}}{\left(\lambda_{\min}^A \right)^{1-\frac{\gamma}{2}}} \right\} \text{ 
    } \forall \text{ } \gamma \in (0,1].
\end{equation*}
\end{corollary}
\begin{proof}
With $ \lambda_{\max}^A \ll 1 \implies \frac{1}{\lambda_{\min}^A} \gg \kappa(A) > 1$ and $\gamma \in (0,1]$
\begin{align}
    \text{CR}_{\LAImath{}} =  1 + \max \left\{\sqrt{\frac{\kappa(A)}{\lambda_{\min}^A}} , \frac{1}{\frac{\lambda_{\min}^A}{2} \cdot \left( \frac{4}{\lambda_{\min}^A} \right)^{\frac{\gamma}{2}} }  \right\} = 1 + \max \left\{ \sqrt{\frac{\kappa(A)}{\lambda_{\min}^A}} , \frac{2^{1-\gamma}}{\left(\lambda_{\min}^A \right)^{1-\frac{\gamma}{2}}} \right\}.
\end{align}
For $\gamma = 1$, this translates to
\begin{align}
    \text{CR}_{\nofi{}} = 1 + \max \left\{ \sqrt{\frac{\kappa(A)}{\lambda_{\min}^A}} , \frac{1}{\left(\lambda_{\min}^A \right)^{\frac{1}{2}}} \right\} = 1 + \sqrt{\frac{\kappa(A)}{\lambda_{\min}^A}}
\end{align}
\end{proof}

\section{General Value function derivative and Proof of Theorem \ref{thm:gen_stochastic}}\label{sec:gen_stochastic}
Before the proof of the theorem, we present an important property of the value function that is applicable to a much broader class of convex hitting costs and switching costs.
\begin{theorem}\label{thm:value_fn_derivative}
Consider stochastic convex hitting costs $f_t(x)$ and switching costs $D_h(x_t||x_{t-1})$ where $h(\cdot)$ is a convex function and $D_h(x_t||x_{t-1})$ is the Bregman divergence and the optimization problem to find the online optimal policy
\begin{align*}
    \argmin_{\{x_t\}_{t=1}^T: x_t \in \mathcal{F}_t} \E \left[\sum_{t=1}^T f_t(x_t) + D_h(x_t||x_{t-1}) \right].
\end{align*}
The value function at round $t$, $V_t(x_{t-1},v_t)$, satisfies the following property
\begin{align*}
    \nabla_{x_{t-1}} V(x_{t-1},v_t) = \nabla^2 h(x_{t-1}) (x_{t-1} - x_t)
\end{align*}
\end{theorem}
\begin{proof}
First, it is clear that the value function at round $t$ is the optimal cost starting from round $t$, means it depends only in the starting state, that is, $(x_{t-1},v_t)$. Further it is obtained by solving the following optimization problem where $\E[V_{t+1}(x,v_{t+1})|\F_t]$ is a function of $v_t$ because of the expectation with respect to $\F_t$.
\begin{align}
    V_{(x_{t-1},v_t)} &= \argmin_x f_t(x) + D_h(x||x_{t-1}) + \E[V_{t+1}(x,v_{t+1})|\F_t]\\
    &= \argmin_x f_t(x) + (h(x) - h(x_{t-1}) - \nabla h(x_{t-1})^T (x - x_{t-1})) + \E[V_{t+1}(x,v_{t+1})|\F_t]
\end{align}
Therefore, we get the optimal action $x_t$ by solving the following differential equation
\begin{align}
    \nabla_x f_t(x) + (\nabla_x h(x) - \nabla h(x_{t-1})) + \nabla_x \E[V(x,v_{t+1}) | \F_{t}] = 0.
\end{align}
Solving this gives $x_t$ as an operator on $x_{t-1}$ and $v_t$. We define the Jacobian of $x_t$ with respect to $x_{t-1}$ as $\left[\frac{\delta (x_t)_j}{\delta (x_{t-1})_i} \right]_{i,j} = J \in \R^{d\times d}$. The optimal cost-to-go at round $t$ is
\begin{align}
    V(x_{t-1},v_t) = f_t(x_t) + D_h(x_t,x_{t-1}) + \E[V(x_t,v_{t+1}) | \F_{t}].
\end{align}
Therefore,
\begin{align}
\begin{split}
    &\nabla_{x_{t-1}} V(x_{t-1},v_t)\\
    =& J \nabla_{x_{t}} f_t(x_t) + J \nabla_{x_{t}} h(x_t) - \nabla h(x_{t-1})\\
    &- (J-I)\nabla h(x_{t-1}) - \nabla^2 h (x_{t-1})(x_t - x_{t-1}) +  J \nabla_{x_{t}} \E[V(x_t,v_{t+1}) | \F_{t}]
\end{split}\\
\begin{split}
    =& J (\nabla_{x_{t}} f_t(x_t) + \nabla_{x_{t}} h(x_t) - \nabla_{x_{t}} h(x_{t-1}) + \nabla_{x_{t}} \E[V(x_t,v_{t+1}) | \F_{t}]) - \nabla^2 h (x_{t-1})(x_t - x_{t-1})
\end{split}\\
=& - \nabla^2 h(x_{t-1})(x_t - x_{t-1})
\end{align}
\end{proof}
We use this in the proof of Theorem \ref{thm:gen_stochastic} where $\nabla^2 h(x_{t-1}) = I$ as our switching cost is $\frac{1}{2}\norm{x_t - x_{t-1}}^2$.
\begin{proof}[Proof of Theorem \ref{thm:gen_stochastic}]
Irrespective of the stochastic process $\{v_t\}_{t=1}^T$, in the final round we know the current minimizer $v_T$ and therefore the optimal action is given by Lemma \ref{lemma: val_fn_T} as $x_T = C_T x_{T-1} + (I-C_T)v_T$. Therefore, the value function derivative at round $T$ is
\begin{align}
    \nabla_{x_{T-1}} V(x_{T-1},v_T) = x_{T-1} - x_T = (I-C_T)x_{T-1} - (I-C_T)v_T
\end{align}
The optimal action at round $T-1$ is obtained by solving
\begin{align}
    A(x_{T-1} - v_{T-1}) + (x_{T-1} - x_{T-2}) + \E[\nabla_{x_{T-1}} V(x_{T-1},v_T)|\F_{T-1}] = 0
\end{align}
which is
\begin{align}
    (2I + A - C_T)x_{T-1} &= x_{T-2} + (I + A - C_T)v_{T-1} + (I-C_T)\E[v_T - v_{T-1}|\F_{T-1}]\\
    C_{T-1}^{-1} x_{T-1} &= x_{T-2} + (C_{T-1}^{-1} - I)v_{T-1} + (I-C_T)\E[v_T - v_{T-1}|\F_{T-1}]
\end{align}
which gives
\begin{align}
    x_{T-1} = C_{T-1} x_{T-2} + (I - C_{T-1})v_{T-1} + C_{T-1}(I-C_T)\E[v_T - v_{T-1}|\F_{T-1}]
\end{align}
proving the theorem for $t = T-1$ and $t=T$ (sum from $T+1$ to $T$ is zero). Suppose the result holds for some $t \in \{1,\ldots,T\}$ then the optimal action at round $t-1$ is a solution of
\begin{align}
    A(x_{t-1} - v_{t-1}) + (x_{t-1} - x_{t-2}) + \E[x_{t-1} - x_t | \F_{t-1}] = 0
\end{align}
or
\begin{align}
\begin{split}
    (2I + A - C_{t})x_{t-1} =& x_{t-2} + Av_{t-1} + (I-C_t)\E[v_t | \F_{t-1}]\\
    &+ \sum_{s=t+1}^T \left(\prod_{q=t}^{s-1} C_q \right) (I-C_s)\E[v_s - v_{s-1}| \mathcal{F}_{t-1}]
\end{split}\\
\begin{split}
    =& x_{t-2} + (I + A - C_t) v_{t-1} + (I-C_t)\E[v_t -v_{t-1}| \F_{t-1}]\\
    &+ \sum_{s=t+1}^T \left(\prod_{q=t}^{s-1} C_q \right) (I-C_s)\E[v_s - v_{s-1}| \mathcal{F}_{t-1}]
\end{split}\\
\begin{split}
    =& x_{t-2} + (C_{t-1}^{-1} - I) v_{t-1} + (I-C_t)\E[v_t -v_{t-1}| \F_{t-1}]\\
    &+ \sum_{s=t+1}^T \left(\prod_{q=t}^{s-1} C_q \right) (I-C_s)\E[v_s - v_{s-1}| \mathcal{F}_{t-1}]
\end{split}
\end{align}
which gives
\begin{align}
\begin{split}
    x_{t-1} =& C_{t-1} x_{t-2} + (I-C_{t-1})v_{t-1} + C_{t-1}(I-C_t)\E[v_t -v_{t-1}| \F_{t-1}]\\
    &+ \sum_{s=t+1}^T C_{t-1}\left(\prod_{q=t}^{s-1} C_q \right) (I-C_s)\E[v_s - v_{s-1}| \mathcal{F}_{t-1}]
\end{split}\\
\begin{split}
    =& C_{t-1} x_{t-2} + (I-C_{t-1})v_{t-1} + \sum_{s=t}^T \left(\prod_{q=t-1}^{s-1} C_q \right) (I-C_s)\E[v_s - v_{s-1}| \mathcal{F}_{t-1}]
\end{split}
\end{align}
proving Theorem \ref{thm:gen_stochastic} by induction.
\end{proof}

\section{Additional proofs in the stochastic environment}

\subsection{Proof of Theorem \ref{thm:hindsight_static_opt}}\label{proof:hindsight_static_opt}
We want the hindsight static optimal for $f_t(x) = \frac{\lambda}{2}\|x-v_t\|_2^2$ and switching cost $\frac{1}{2}\|x_t - x_{t-1}\|_2^2$ with the minimizers evolve as martingale. We write the minimizer at $t$ as a sum of the initial point $x_0$ and martingale differences $u_k = v_k - v_{k-1}$ with $v_0 = x_0$, described below
\begin{align}
    v_t = x_0 + \sum_{k=1}^t u_k.
\end{align}
The martingale differences $u_k$ are assumed to have a variance of at least $\sigma^2>0$. The hindsight optimal minimizing the total cost over the horizon is
\begin{align}
    x^* = \argmin_x \frac{1}{2}\|x - x_0\|_2^2 + \sum_{t=1}^T \frac{\lambda}{2}\|x-v_t\|_2^2.
\end{align}
The optimal solution will satisfy the following expression, obtained by differentiating the objective function,
\begin{align}
     (1 + \lambda T)x^*  &= x_0 + \lambda \sum_{t=1}^T v_t\\
\end{align}
which gives
\begin{align*}
     x^* &= \left(\frac{x_0}{1 + \lambda T} \right) + \left(\frac{\lambda T}{1 + \lambda T}\right) \left( \frac{\sum_{t=1}^T v_t}{T}\right)\\
     &= \left(\frac{x_0}{1 + \lambda T} \right) + \left(\frac{\lambda T}{1 + \lambda T}\right) \left( x_0 + \frac{\sum_{t=1}^T \sum_{k=1}^t u_k}{T}\right)\\
     &= x_0 + \left(\frac{\lambda T}{1 + \lambda T}\right) \left(\frac{\sum_{t=1}^T \sum_{k=1}^t u_k}{T}\right)\\
     &= x_0 + \left(\frac{\lambda T}{1 + \lambda T}\right) \left(\frac{\sum_{k=1}^T u_k \sum_{t=k}^T 1 }{T}\right)\\
     &= x_0 + \left(\frac{\lambda T}{1 + \lambda T}\right) \left(\frac{\sum_{k=1}^T (T-k+1) u_k}{T}\right)\\
      &= x_0 + \sum_{k=1}^T 
 \frac{\lambda(T-k+1) u_k}{1 + \lambda T}
\end{align*}
This implies
\begin{align}
    x^* - x_0 &= \sum_{k=1}^T 
 \frac{\lambda(T-k+1) u_k}{1 + \lambda T}\\
 x^* - v_t &= \sum_{k=1}^T \frac{\lambda(T-k+1) u_k}{1 + \lambda T} - \sum_{k=1}^t u_k.
\end{align}
We can simplify $x^* - v_t$ to a more usable expression,
\begin{align}
    x^* - v_t &= \sum_{k=1}^T \frac{\lambda(T-k+1) u_k}{1 + \lambda T} - \sum_{k=1}^t u_k\\
    &= \sum_{k=t+1}^T \frac{\lambda(T-k+1) u_k}{1 + \lambda T} + \sum_{k=1}^t \left( \frac{\lambda(T-k+1)}{1 + \lambda T} - 1 \right) u_k\\
    &= \sum_{k=t+1}^T \left(\frac{\lambda(T-k+1)}{1 + \lambda T}\right)u_k - \sum_{k=1}^t \left( \frac{\lambda(k-1) + 1}{1 + \lambda T}\right) u_k\\
\end{align}
Therefore, the total expected cost of using $x^*$ is
\begin{align}
    &\E[\text{Cost}_{x^*}[1,T]] - \frac{1}{2}\E\norm{x^* - x_0}^2 \\
    &= \sum_{t=1}^T \frac{\lambda}{2}\E\norm{x^*-v_t}^2\\
    &= \sum_{t=1}^T \frac{\lambda}{2}\E\norm{\sum_{k=t+1}^T \left(\frac{\lambda(T-k+1)}{1 + \lambda T}\right)u_k - \sum_{k=1}^t \left( \frac{\lambda(k-1) + 1}{1 + \lambda T}\right) u_k}^2\\
    &= \sum_{t=1}^T \frac{\lambda}{2(1 + \lambda T)^2} \E \norm{\sum_{k=t+1}^T \lambda(T-k+1)u_k - \sum_{k=1}^t (\lambda(k-1) + 1) u_k}^2 \\
\end{align}
Now consider $i>j$ for the following,
\begin{align}
    \E[u_j^Tu_i] &= \E[\E[u_j^T u_i | \mathcal{F}_{i-1}]]\\
    &= \E[u_j^T\E[ u_i | \mathcal{F}_{i-1}]]\\
    &= \E[u_j^T 0]
    = 0.
\end{align}
This gives us
\begin{align}
    &\E[\text{Cost}_{x^*}[1,T]] - \frac{1}{2}\E \norm{x^* - x_0}^2 \\
    &= \frac{\lambda}{2(1 + \lambda T)^2}\sum_{t=1}^T \left(\sum_{k=t+1}^T \lambda^2(T-k+1)^2\sigma^2 + \sum_{k=1}^t (\lambda(k-1) + 1)^2 \sigma^2 \right)\\
    &= \frac{\lambda}{2(1 + \lambda T)^2}\left( \sum_{t=1}^T \sum_{k=t+1}^T \lambda^2(T-k+1)^2\sigma^2 + \sum_{t=1}^T \sum_{k=1}^t (\lambda(k-1) + 1)^2 \sigma^2 \right)\\
    &= \frac{\lambda}{2(1 + \lambda T)^2}\left( \sum_{k=2}^T \lambda^2(T-k+1)^2\sigma^2 \sum_{t=1}^{k-1} 1 + \sum_{k=1}^T (\lambda(k-1) + 1)^2 \sigma^2 \sum_{t=k}^T 1 \right)\\
    &= \frac{\lambda}{2(1 + \lambda T)^2}\left( \sum_{k=2}^T (k-1)\lambda^2(T-k+1)^2\sigma^2 + \sum_{k=1}^T (T-k+1) (\lambda(k-1) + 1)^2 \sigma^2  \right)\\
    &= \frac{\lambda \sigma^2}{2(1 + \lambda T)^2}\sum_{k=2}^T (T-k+1)\bigg[\lambda^2 (k-1) (T-k+1) + (\lambda(k-1) + 1)^2\bigg] + \frac{T\lambda \sigma^2}{2(1 + \lambda T)^2} \\
    &= \frac{\lambda \sigma^2}{2(1 + \lambda T)^2}\sum_{k=2}^T (T-k+1)\bigg[\lambda^2 (k-1)T - \lambda^2 (k-1)^2 + (\lambda(k-1) + 1)^2\bigg] + \frac{T\lambda \sigma^2}{2(1 + \lambda T)^2} \\
    &= \frac{\lambda \sigma^2}{2(1 + \lambda T)^2}\sum_{k=2}^T (T-k+1)\bigg[\lambda^2 (k-1)T +2\lambda(k-1) + 1\bigg] + \frac{T\lambda \sigma^2}{2(1 + \lambda T)^2} \\
    &= \frac{\lambda \sigma^2}{2(1 + \lambda T)^2}\sum_{k=2}^T (T-k+1)\bigg[\lambda (k-1) (\lambda T +2) + 1\bigg] + \frac{T\lambda \sigma^2}{2(1 + \lambda T)^2} \\
    &= \frac{\sigma^2\lambda^2 (\lambda T +2) }{2(1 + \lambda T)^2}\sum_{k=2}^T (T-k+1)(k-1)+ \frac{T\lambda \sigma^2}{2(1 + \lambda T)^2} + \frac{\lambda \sigma^2}{2(1 + \lambda T)^2}\sum_{k=2}^T (T-k+1) \\
    &= \frac{\sigma^2\lambda^2 (\lambda T +2) }{2(1 + \lambda T)^2} \left( T\sum_{k=2}^T (k-1) - \sum_{k=2}^T (k-1)^2 \right) + \frac{\lambda \sigma^2}{2(1 + \lambda T)^2}\left(\frac{T(T+1)}{2} \right)\\
    &=\frac{\sigma^2\lambda^2 (\lambda T +2) }{2(1 + \lambda T)^2} \left( \frac{T^2 (T-1)}{2} - \frac{(T-1)T(2T-1)}{6} \right) + \frac{\lambda \sigma^2}{2(1 + \lambda T)^2}\left(\frac{T(T+1)}{2} \right)\\
    &=\frac{\sigma^2\lambda^2 (\lambda T +2) }{2(1 + \lambda T)^2} \left( \frac{(T-1)T(T+1)}{6}\right) + \frac{\lambda \sigma^2}{2(1 + \lambda T)^2}\left(\frac{T(T+1)}{2} \right)\\
    &=\frac{\sigma^2\lambda^2}{12} \left( \frac{(\lambda T +2)(T-1)T(T+1)}{(1 + \lambda T)^2}\right) + \frac{\lambda \sigma^2}{2(1 + \lambda T)^2}\left(\frac{T(T+1)}{2} \right)\\
    &= \Theta(T^2)
\end{align}
\[
\pushQED{\qed} 
\qedhere
\popQED
\]    

\subsection{Properties of $\{C_t\}_t$}\label{proof:Ct_main_props}
\begin{lemma}\label{lemma:Ct_main_props}
For any $t \in \{1, \ldots, T\}$, the matrix $C_t$ satisfies the following properties:
\begin{enumerate}[{\normalfont(a)}]
    \item $C_t$ is invertible
    \item $C_t$ has the same set of eigenvectors as $A$.
    \item Eigenvalues of $C_t$ lie in $(0,1)$, and hence, $C_t$ is positive definite.
\end{enumerate}
\end{lemma}
\begin{proof}
We prove this result using induction. Because $A$ is a positive definite matrix, its eigenvalues are greater than 1. Consequently, $(I+A)$ is invertible and eigenvalues of $C_T = (I+A)^{-1}$ are in $(0,1)$. Further, $(I+A)^{-1}$ has the same set of eigenvectors as $A$. Therefore, the lemma holds for $t = T$. Now assume the claim for $t \leq T$. From the recursion relation, we have
\begin{align}
    C_{t-1}^{-1} = 2I + A - C_{t}
\end{align}
Now, $2I + A - C_{t}$ will have the same eigenvectors as $A$. Further, eigenvalues of $2I + A - C_{t}$ are $\{2 + \lambda_i^A - \lambda_i^t \}_t$ and greater than 1 as $\{\lambda_i^t\}_t$, the eigenvalues of $C_{t}$, are in $(0,1)$. Therefore, $2I + A - C_{t}$ is invertible and so is $C_{t-1}$. Consequently, $C_{t-1}$ has same eigenvectors as that of $A$ and the eigenvalues of $C_{t-1}$ are in $(0,1)$.
\end{proof}

\subsection{Asymptotic behavior of $\{C_t\}_{t=1}^T$}\label{proof:ct and Ct properties}

\begin{lemma}\label{lemma:scalar_ct_props}
For any horizon $T>0$, consider the sequence of numbers $\{c_t\}_{t=1}^T$ satisfying
\begin{equation*}
    \frac{1}{c_t} = 2 + \lambda - c_{t+1}
\end{equation*}
with $c_T = \frac{1}{\lambda + 1}$ for some $\lambda > 0$. Such a sequence of numbers has the following properties
\begin{enumerate}
    \item \label{lemma:scalar_ct_props_part_1} For $A=\lambda I$, the \ai{} algorithm will have
    \begin{equation*}
        C_t = c_t I \text{ } \forall \text{ } t \in \{1,\ldots, T\}
    \end{equation*}
    \item \label{lemma:scalar_ct_props_part_2} The evolution of $c_t$ is
    \begin{equation*}
        (c_t - c_L) = (c_{t+1} - c_L)c_t c_L
    \end{equation*}
    \item \label{lemma:scalar_ct_props_part_3} $\{c_t\}_{t=1}^T$ is an increasing sequence that is
    \begin{equation*}
      c_L < c_1 < \ldots < c_T
    \end{equation*}
    \item \label{lemma:scalar_ct_props_part_4} $\{c_t\}$ satisfy
    \begin{equation*}
        (c_t - c_L) \leq (1 - c_L)\left(c_T\right)^{2(T-t+1)}
    \end{equation*}
    \item \label{lemma:scalar_ct_props_part_5} The behavior of $c_1$ as horizon goes to infinity is
    \begin{equation*}
        \lim_{T \to \infty} c_1 = c_L
    \end{equation*}
\end{enumerate}
where $c_L = \frac{\lambda + 2 -\sqrt{\lambda^2 + 4\lambda}}{2} = \frac{2}{\lambda + 2 + \sqrt{\lambda^2 + 4\lambda}} \in (0,1)$ for $\lambda > 0$.
\end{lemma}

\begin{proof}
For the part \ref{lemma:scalar_ct_props_part_1} of the lemma, we prove it by induction. First observe that
\begin{align}
    C_T &= (A+I)^{-1}\\
    &= \left(\frac{1}{\lambda + 1} \right) I
\end{align}
for $\lambda > 0$. This means part (i) holds for $t = T$. Assuming the claim for some $t \leq T$, that is, $C_t = c_t I$
\begin{align}
    C_{t-1} &= (2I + A - C_t)^{-1}\\
    &= ((1+\lambda - c_t)I)^{-1}\\
    &= \left(\frac{1}{2 + \lambda - c_t} \right) I\\
    &= c_{t-1} I
\end{align}
Now, observe that
\begin{align}
    c_L + \frac{1}{c_L} &= \frac{\lambda + 2 -\sqrt{\lambda^2 + 4\lambda}}{2} + \frac{\lambda + 2 +\sqrt{\lambda^2 + 4\lambda}}{2}\\
    &= \lambda + 2
\end{align}
and to recap $\{c_t\}_{t=1}^T$ satisfy the recursion of Lemma \ref{lemma:scalar_ct_props}, that is
\begin{align}\label{eqn:scalar_coeff_recurr}
    \frac{1}{c_t} = 2+\lambda -c_{t+1}
\end{align}
and $c_T = \frac{1}{\lambda + 1}$. Therefore,
\begin{align}\label{eqn:coeff_recurr_alt}
    \frac{1}{c_L} - \frac{1}{c_t} = c_{t+1} - c_L
\end{align}
which gives
\begin{align}
    (c_{t} - c_L) = (c_{t+1} - c_L) c_t c_L
\end{align}
proving part \ref{lemma:scalar_ct_props_part_2} of the lemma. Now based on the recursion \eqref{eqn:scalar_coeff_recurr}, 
\begin{align}
    c_{t+1} \in (0,1) \implies c_t \in (0,1)
\end{align}
and $c_T = \frac{1}{\lambda + 1} \in (0,1)$, which inductively proves that $c_t \in (0,1)$  for all $t$. Now, observe that $c_L \in (0,1)$ for $\lambda > 0$ and hence,
\begin{align}
    \begin{split}
    (c_{t} - c_L) &< (c_{t+1} - c_L)\\
    c_t &< c_{t+1}
    \end{split}
\end{align}
Also notice that because of the recursion \eqref{eqn:coeff_recurr_alt}, $$c_{t+1} > c_L \implies c_t > c_L$$ and we know 
\begin{align}
    c_L = \frac{2}{(\lambda+1) + (1 + \sqrt{\lambda^2 + 4\lambda})}  < \frac{1}{\lambda + 1} = c_T \text{  } \forall \text{ } \lambda>0
\end{align} which inductively proves that $c_t >c_L$  for all $t$ completing the proof of part \ref{lemma:scalar_ct_props_part_3} of the lemma. 
Observe that
\begin{align}
    (c_t - c_L) &= (c_{t+1} - c_L)c_t c_L\\
    &\leq (c_{t+1} - c_L)c_T^2\\
    &\leq (c_T - c_L)c_T^{2(T-t)}
\end{align}
Also, note that 
\begin{align}
    \frac{1}{c_T} = 2 + \lambda - 1 \implies (c_T - c_L) = (1-c_L)c_Tc_L
\end{align}
and therefore,
\begin{align}
    (c_t - c_L) &\leq (1 - c_L)c_T^{2(T-t+1)}
\end{align}
proving part \ref{lemma:scalar_ct_props_part_4} of the lemma. Finally because $c_T < 1$ for $\lambda > 0$,
\begin{align}
    \lim_{T \to \infty} c_1 \leq  \lim_{T \to \infty} (c_T - c_L)c_T^{2T} = 0
\end{align}
proving part \ref{lemma:scalar_ct_props_part_5} of the lemma.
\end{proof}
\begin{definition}
The square root of a real positive definite matrix $M$ is the unique positive definite matrix $M^{1/2}$ such that $$M^{1/2}\left(M^{1/2}\right)^T = \left(M^{1/2}\right)^TM^{1/2} = M.$$ Since, $M$ can be written as $PD_M P^T$ where $D_M$ is the diagonal matrix of eigenvalues of $M$ and $P$ is the matrix having eigenvectors of $M$ as columns (also called the \textit{modal matrix} of $A$), $$M^{1/2} = PD_M^{1/2} P^T,$$  and $D_M^{1/2}$ is the diagonal matrix of square root of eigenvalues of $M$.
\end{definition}
\begin{definition}\label{def:CL_defn}
Define $C_L$ as
\begin{equation*}
    C_L := \frac{A+2I - \sqrt{A^2 + 4A}}{2}.
\end{equation*}
Then $C_L$ will have the same set of eigenvectors as $A$, with eigenvalues $\left\{ \lambda_i^L \right\}_{i=1}^d$ where
\begin{equation*}
    \lambda_i^L = \frac{\lambda_i^A + 2 - \sqrt{\left(\lambda_i^A\right)^2 + 4 \lambda_i^A}}{2} = \frac{2}{\lambda_i^A + 2 + \sqrt{\left(\lambda_i^A\right)^2 + 4 \lambda_i^A}}.
\end{equation*}
Note that $\lambda_i^L$ is a decreasing function of $\lambda_i^A$ and therefore,
\begin{align}
    \lambda_{\min}^L &= \frac{\lambda_{\max}^A + 2 - \sqrt{\left(\lambda_{\max}^A\right)^2 + 4 \lambda_{\max}^A}}{2}\\
    \lambda_{\max}^L &= \frac{\lambda_{\min}^A + 2 - \sqrt{\left(\lambda_{\min}^A\right)^2 + 4 \lambda_{\min}^A}}{2}
\end{align}
\end{definition}

\begin{corollary}\label{corr:Ct_add_props}
For any horizon $T>0$, consider a matrix sequence $\{M_t\}_{t=1}^T$ satisfying
\begin{align}
    M_t^{-1} = 2I + A - M_{t+1} \text{ } \forall \text{ } t \in \{1,\ldots, T-1\}
\end{align}
with $C_L \prec M_T \prec I$ and having eigenvectors same as $A$. Denoting the eigenvalues of $\{M_t\}_{t=1}^T$ corresponding to the $i^{th}$ eigenvector as $\{\rho_i^t\}_{t=1}^T$, they satisfy the following properties:
\begin{enumerate}
    \item \label{corr:Ct_add_props_part_1} The $i^{th}$ eigenvalue of the matrices $\{M_t\}_{t=1}^T$ satisfy
    \begin{equation*}
        \frac{1}{\rho_i^t} = 2 + \lambda_i^A - \rho_i^{t+1} \text{ } \forall \text{ } t \in \{1,\ldots,T-1\}
    \end{equation*}
    \item \label{corr:Ct_add_props_part_2} The $i^{th}$ eigenvalue of the matrices $\{M_t\}_{t=1}^T$ satisfy 
    \begin{equation*}
        \left( \rho_i^t - \lambda_i^L \right) = \left( \rho_i^{t+1} - \lambda_i^L \right) \rho_i^t \lambda_i^L
    \end{equation*}
    \item \label{corr:Ct_add_props_part_3} The following order holds for any horizon $T$
    \begin{equation*}
        C_L \prec M_1 \prec \ldots \prec M_T
    \end{equation*}
    \item \label{corr:Ct_add_props_part_4} The $i^{th}$ eigenvalues of $\{M_t\}_{t=1}^T$ show the following behavior
    \begin{equation*}
        \left( \rho_i^t - \lambda_i^L \right) \leq \left(\rho_i^T- \lambda_i^L \right) \left(\rho_i^T \right)^{2(T-t)}
    \end{equation*}
    \item \label{corr:Ct_add_props_part_5} The asymptotic behavior of $\{M_t\}_{t=1}^T$ is
    \begin{equation*}
        \lim_{T \to \infty}     z^T(M_1 - C_L)z = 0 \text{ }  \forall \text{ } z \in \R^d.
    \end{equation*}
\end{enumerate}
\end{corollary}
\begin{proof}
Consider the following matrix
\begin{align}
    C_L = \frac{A + 2I - \sqrt{A^2 + 4A}}{2}
\end{align}
where $\sqrt{A^2 + 4A}$ is the unique positive definite square root of the positive definite matrix $(A^2 + 4A)$. We know that the square root has to be $P D^{1/2} P^T$ where $D$ is the diagonal matrix of eigenvalues of $(A^2 + 4A)$ and $P$ is the matrix having the corresponding eigenvectors of $(A^2 + 4A)$ as columns. Therefore, $\sqrt{A^2 + 4A}$ has the same set of eigenvectors as $(A^2 + 4A)$. This means that $C_L$ has the same set of eigenvectors as $A$. This means that
\begin{align}
    \lambda_i^L = \frac{\lambda_i^A + 2 - \sqrt{\left(\lambda_i^A \right)^2 + 4\lambda_i^A }}{2}  = \frac{2}{\lambda_i^A + 2 + \sqrt{\left(\lambda_i^A \right)^2 + 4\lambda_i^A }} < \rho_i^T
\end{align}
where $\lambda_i^L$ is the $i^{th}$ eigenvalue of $C_L$ and $\lambda_i^A$ is the $i^{th}$ eigenvalue of $A$. Now, consider the recursion followed by $\{C_t\}_{t=1}^T$,
\begin{align}
    M_{t}^{-1} &= 2I + A - M_{t+1}
\end{align}
and we know that each one of $\{M_t\}_{t=1}^T$ has same the set of eigenvectors as $A$. This gives us the following recursion followed by the $i^{th}$ eigenvalue of $\{C_t\}_{t=1}^T$
\begin{align}\label{eqn:eigenval_recurr}
    \frac{1}{\rho_i^{t}} &= 2 + \lambda_i^A - \rho_i^{t+1}
\end{align}
Now, from the value of $\lambda_i^L$, we know
\begin{align}
    \lambda_i^L + \frac{1}{\lambda_i^L} = 2 + \lambda_i^A.
\end{align}
which when substituted in \eqref{eqn:eigenval_recurr} gives
\begin{align}
    \left( \rho_i^t - \lambda_i^L \right) = \left( \rho_i^{t+1} - \lambda_i^L \right) \rho_i^t \lambda_i^L
\end{align}
Now observe that $\left\{ \rho_i^t \right\}_{t=1}^T$ follows recursion \eqref{eqn:scalar_coeff_recurr} with $\lambda = \lambda_i^A$. Further, $\lambda_i^L$ is $c_L$ with $\lambda = \lambda_i^A$. Therefore, applying the properties in Lemma \ref{lemma:scalar_ct_props}, we get that
\begin{align}
    \lambda_i^L < \rho_i^1 < \ldots < \rho_i^T \text{ } \forall \text{ } i \in \{1,\ldots, d\}
\end{align}
and since, $C_L$ and each of $\{M_t\}_t$ have the same set of eigenvectors,
\begin{align}
    C_L \prec M_1 \prec M_2 \prec \ldots \prec M_T.
\end{align}
Further,
\begin{align}
    (\rho_i^{t} - \lambda_i^{L}) &\leq (\rho_i^{t+1} - \lambda_i^{L}) \cdot (\rho_i^{T})^2\\
    &\leq  (\rho_i^{T} - \lambda_i^{L}) \cdot (\rho_i^{T})^{2(T-t)}
\end{align}
Since $\rho_i^{T} < 1$ for every $i$, we have
\begin{align}
    \lim_{T \to \infty} \left(\rho_i^{1} -\lambda_i^{L} \right) \leq \lim_{T \to \infty} (\rho_i^{T} - \lambda_i^{L}) \cdot (\rho_i^{T})^{2(T-1)} = 0 \text{ } \forall \text{ } i \in \{1,\ldots, d\}.
\end{align}
Now, consider $M_1$ and $C_L$. We know $M_1 = PD_1P^T$ and $PD_LP^T$ where $D_1 = diag\left(\left(\rho_i^{1}\right)_{i=1}^d \right)$,  $D_L = diag\left(\left(\lambda_i^{L}\right)_{i=1}^d \right)$ and $P$ is the modal matrix of $A$.
\begin{align}
    z^T(M_1 - C_L)z &= z^TP(M_1 - D_L)P^T z\\
    &= \begin{bmatrix}
        k_1 & \ldots & k_d
    \end{bmatrix}
    \begin{bmatrix}
        \rho_1^1 - \rho_1^L & 0 & \ldots & 0\\
        0 & \rho_2^1 - \rho_2^L & \ldots & 0\\
        \vdots & \ddots & \ldots & \vdots\\
        0 & \ldots & \ldots & \rho_d^1 - \lambda_d^L
    \end{bmatrix}
    \begin{bmatrix}
        k_1\\
        \vdots\\
        k_d
    \end{bmatrix}\\
    &= \sum_{i=1}^d \left(\rho_i^{1} -\lambda_i^{L} \right) k_i^2
\end{align}
for any $z \in \R^d$. Therefore, in the case of infinite horizon,
\begin{align}
    \lim_{T \to \infty}     z^T(M_1 - C_L)z = 0 \text{ }  \forall \text{ } z \in \R^d
\end{align}
\end{proof}     

\subsection{Proof of Theorem \ref{thm:R-OBD_lin_regret}}\label{proof:R-OBD_alg_and_R-OBD_lin_regret}
The \robd{} solution for hitting costs $f_t(x) = \frac{1}{2}(x-v_t)^T A (x-v_t)$ and switching costs $c(x_t,x_{t-1}) = \frac{1}{2}\|x_t - x_{t-1}\|_2^2$ is
\begin{align}
    x_t = \argmin_x \frac{1}{2} (x-v_t)^T A (x-v_t) + \frac{\mu_1}{2} \|x - x_{t-1}\|_2^2 + \frac{\mu_2}{2}\|x - v_t\|_2^2.
\end{align}
Equating the derivative of the objective to zero
\begin{align}
    (A + \mu_1 I + \mu_2 I)x = \mu_1 x_{t-1} + (A+\mu_2)v_t.
\end{align}
Setting $\mu_1 = 1$, we get
\begin{align}
    x_t &= (A+I+\mu_2 I)^{-1} x_{t-1} + (I - (A+I+\mu_2 I)^{-1}) v_t\\
    &= C_{\robd{}} x_{t-1} + (I - C_{\robd{}}) v_t
\end{align}
with $\mu_2$ as
\begin{align}
    \mu_2 = \frac{\lambda_{\min}^A}{2} \left( \sqrt{1 + \frac{4}{\lambda_{\min}^A}} - 1 \right).
\end{align}
Now, lets have a look at eigenvalues of $C_{\robd{}}$
\begin{align}
    C_{\robd{}} &= (A+I+\mu_2 I)^{-1}\\
    &= P (D_A+I+\mu_2 I)^{-1} P^T
\end{align}
which gives us
\begin{align}
    \lambda_i^{C_{\robd{}}} &= \frac{1}{\lambda_i^A + 1 +\frac{\sqrt{\left(\lambda_{\min}^A\right)^2 + 4\lambda_{\min}^A} - \lambda_{\min}^A}{2}}
\end{align}
Therefore, $A \neq \lambda I$ implies there exists $i$ such that $\lambda_i^A \neq \lambda_{\min}^A$, giving us $C_{\robd{}} \neq C_L$. Now,
\begin{align}
    \lambda_i^{C_{\robd{}}} = \lambda_i^L \iff \lambda_i^A = \lambda_{\min}^A
\end{align}
and therefore, we can replace the condition $(C_{\robd{}} - C_L)\Sigma \neq 0_{d \times d}$ and $$Z = \{ i \in \{1,\ldots,d\}: \lambda_i^{C_{\robd{}}} = \lambda_i^{L} \}$$ with $(A - \lambda_{\min}^A I)\Sigma \neq 0_{d \times d}$ and $$Z = \{ i \in \{1,\ldots,d\}: \lambda_i^A = \lambda_{\min}^A \}$$
which proves the theorem through Theorem \ref{thm:FI_lin_regret}.
\[
\pushQED{\qed} 
\qedhere
\popQED
\]

\subsection{Linear regret for Follow the Minimizer (\ftm{}) algorithm}\label{proof:FtM_regret}
\begin{lemma}\label{corr:FtM_lin_regret_exact}
For $f_t(x) = \frac{1}{2} (x-v_T)^T A (x-v_t)$ and $c(x_t,x_{t-1}) = \frac{1}{2}\|x_t - x_{t-1}\|_2^2$ with minimizers $\{v_t\}_t$ being a martingale, the \ftm{} algorithm has regret with respect to \ai{} as
\begin{equation*}
    \text{Regret}_{\ftm{}}[1,T] \geq \left(\frac{\lambda_{\min}^L \sigma^2}{2}\right)T
\end{equation*}
\end{lemma}
\begin{proof}
The cost of \ftm{} algorithm for $f_t(x) = \frac{1}{2}(x-v_t)^T A (x-v_t)$ will be
\begin{align}
    \E[\text{Cost}_{\ftm{}}[1,T]] = \E\left[0 + \frac{1}{2}\|x_0 - v_1\|_2^2 + \sum_{t=2}^T \left(0 + \frac{1}{2}\|v_t - v_{t-1}\|_2^2 \right)\right]
\end{align}
and, therefore, the gap to \ai{} will be
\begin{align}
\begin{split}
\E[\text{Cost}_{\ftm{}}[1,T]] - \E[\text{Cost}_{\text{\ai{}}}[1,T]] =& \E\left[ \frac{1}{2}(x_{0}-v_{1})^T (x_{0}-v_{1}) \bigg| \mathcal{F}_{0} \right]\\
&+ \sum_{t=2}^{T} \E\left[ \frac{1}{2}(v_t - v_{t-1})^T (v_{t} - v_{t-1}) \bigg| \mathcal{F}_{0} \right] \\
&-\E\left[ \frac{1}{2}(x_{0}-v_{1})^T (I - C_{1}) (x_{0}-v_{1}) \bigg| \mathcal{F}_{0} \right]\\
&-\sum_{t=2}^{T} \E\left[ \frac{1}{2}(v_t - v_{t-1})^T (I-C_{t}) (v_{t} - v_{t-1}) \bigg| \mathcal{F}_{0} \right]
\end{split}\\
\begin{split}
=&\E\left[ \frac{1}{2}(x_{0}-v_{1})^T C_{1} (x_{0}-v_{1}) \bigg| \mathcal{F}_{0} \right]\\
&+\sum_{t=2}^{T} \E\left[ \frac{1}{2}(v_t - v_{t-1})^T C_{t} (v_{t} - v_{t-1}) \bigg| \mathcal{F}_{0} \right]
\end{split}\\
\begin{split}
\geq&\E\left[ \frac{1}{2}(x_{0}-v_{1})^T C_{L} (x_{0}-v_{1}) \bigg| \mathcal{F}_{0} \right]\\
&+\sum_{t=2}^{T} \E\left[ \frac{1}{2}(v_t - v_{t-1})^T C_{L} (v_{t} - v_{t-1}) \bigg| \mathcal{F}_{0} \right]
\end{split}\\
\begin{split}
\geq&\E\left[ \frac{1}{2}(x_{0}-v_{1})^T \lambda_{\min}^L I (x_{0}-v_{1}) \bigg| \mathcal{F}_{0} \right]\\
&+\sum_{t=2}^{T} \E\left[ \frac{1}{2}(v_t - v_{t-1})^T \lambda_{\min}^L I (v_{t} - v_{t-1}) \bigg| \mathcal{F}_{0} \right]
\end{split}\\
=& \left(\frac{\lambda_{\min}^L \sigma^2}{2}\right)T
\end{align}
\end{proof}

\section{Additional proofs for \ai{}'s adversarial performance}
\subsection{Proof of Theorem \ref{thm:AI_cr}}\label{proof:AI_cr}
The \ai{} algorithm is
\begin{align}
    x_t = C_t x_{t-1} + (I-C_t) v_t
\end{align}
which means
\begin{align}
    A(x_t - v_t) + (x_t - x_{t-1}) + (C_t^{-1} - A - I)(x_t - v_t) = 0.
\end{align}
The above equation can be written as
\begin{align}
    x_t = \argmin_x \frac{1}{2}(x_t - v_t)^T A (x_t - v_t) +  \frac{1}{2} \|x_t - x_{t-1}\|_2^2 + \frac{1}{2}(x_t - v_t)^T(C_t^{-1} - (I+A))(x_t - v_t).
\end{align}
This means $m = \lambda_{\min}^A$ and $\alpha = \beta = 1$. Remember,
\begin{align}
    C_L^{-1} \succ C_1^{-1} \succ \ldots \succ C_T^{-1}
\end{align}
which means
\begin{align}
    C_L^{-1} - (I+A) \succ C_1^{-1} - (I+A) \succ \ldots \succ C_T^{-1} - (I+A) = 0
\end{align}
Now, $\alpha'_t$ will be the smallest eigenvalue of $C_t^{-1} - (I+A)$ and $\beta'_t$ will be the largest eigenvalue of $C_t^{-1} - (I+A)$. Therefore, $\min_t \alpha'_t = \alpha'_T = 0$ and $\max_t \beta'_t = \beta'_1$. Now, the largest eigenvalue of $C_1^{-1} - (I+A)$ is less than that of of $C_L^{-1} - (I+A)$, that is
\begin{align}
    \beta'_1 &< \max_i \left\{\frac{1}{\lambda_i^L}-(1+\lambda_i^A) \right\}\\
    &= \max_i \frac{\sqrt{\left(\lambda_i^A\right)^2 + 4\lambda_i^A} - \lambda_i^A}{2}\\
    &=\frac{\sqrt{\left(\lambda_{\max}^A\right)^2 + 4 \lambda_{\max}^A} - \lambda_{\max}^A}{2}
\end{align}
This gives us the competitive ratio of dynamic programming for any horizon $T$ as
\begin{align}
    \text{CR}_{\text{\ai{}}} &= \max \left\{1 + \frac{\sqrt{\left(\lambda_{\max}^A\right)^2 + 4 \lambda_{\max}^A} - \lambda_{\max}^A}{2 \lambda_{\min}^A} , 1 + \frac{1}{\lambda_{\min}^A}  \right\}\\
\end{align}
Now the $(\sqrt{y^2 + 4y} - y) < 2$ for all $y\geq 0$, which means the second term is always larger than the first and
\begin{align}
    \text{CR}_{\text{\ai{}}} = 1 + \frac{1}{\lambda_{\min}^A}
\end{align}
\[
\pushQED{\qed} 
\qedhere
\popQED
\]

\section{Additional proofs regarding \LAItxt{}}

\subsection{Proof of Lemma \ref{lemma:nofi_dp_eigenvalues}}\label{proof:nofi_dp_eigenvalues}
As a consequence of $\{C_t\}_t$ and $\{\Tilde{C}_t\}$ having eigenvectors same as that of $A$ and satisfying
\begin{align}
    C_t^{-1} &= 2I + A - C_{t+1}\\
    \Tilde{C}_t^{-1} &= 2I + A - \Tilde{C}_{t+1},
\end{align}
the eigenvalues of $C_t$ and $\Tilde{C}_t$ follow the recursion
\begin{align}
    \frac{1}{\lambda_i^t} &= 2 + \lambda_i^A - \lambda_i^{t+1}\\
    \frac{1}{\Tilde{\lambda}_i^t} &= 2 + \lambda_i^A - \Tilde{\lambda}_i^{t+1}
\end{align}
for $t \in \{1,\ldots, T-1\}$ . This means
\begin{align}
    \frac{1}{\Tilde{\lambda}_i^t} - \frac{1}{\lambda_i^t} = \lambda_i^{t+1} - \Tilde{\lambda}_i^{t+1}
\end{align}
or
\begin{align}
    \left(\lambda_i^t - \Tilde{\lambda}_i^t \right) &= \left(\lambda_i^{t+1} - \Tilde{\lambda}_i^{t+1} \right)\lambda_i^t\Tilde{\lambda}_i^t\\
    &\leq \left(\lambda_i^{t+1} - \Tilde{\lambda}_i^{t+1} \right)\lambda_i^T\Tilde{\lambda}_i^T\\
    &\leq \left(\lambda_i^{t+1} - \Tilde{\lambda}_i^{t+1} \right)\left(\lambda_i^T \right)^2\\
    &\leq \left(\lambda_i^{T} - \Tilde{\lambda}_i^{T} \right)\left(\lambda_i^T \right)^{2(T-t)}
\end{align}
Further,
\begin{align}
    \frac{1}{\lambda_i^{T}} &= 2 + \lambda_i^A - 1\\
    \frac{1}{\Tilde{\lambda}_i^{T}} &= 2 + \lambda_i^A -\left( 1- \frac{\lambda_i^A}{2}\left( \left(1 + \frac{4}{\lambda_i^A}\right)^{\frac{\gamma}{2}} - 1 \right) \right)
\end{align}
meaning
\begin{align}
    \lambda_i^{T} - \Tilde{\lambda}_i^{T} &= \frac{\lambda_i^A}{2}\left( \left(1 + \frac{4}{\lambda_i^A}\right)^{\frac{\gamma}{2}} - 1 \right)\lambda_i^{T}\Tilde{\lambda}_i^{T}\\
    &\leq \frac{\lambda_i^A}{2}\left( \left(1 + \frac{4}{\lambda_i^A}\right)^{\frac{\gamma}{2}} - 1 \right)\left(\lambda_i^T \right)^2.
\end{align}
This gives
\begin{align}
    \left(\lambda_i^t - \Tilde{\lambda}_i^t \right) &\leq \frac{\lambda_i^A}{2}\left( \left(1 + \frac{4}{\lambda_i^A}\right)^{\frac{\gamma}{2}} - 1 \right) \left(\lambda_i^T \right)^{2(T-t+1)}
\end{align}
\[
\pushQED{\qed} 
\qedhere
\popQED
\]    

\begin{lemma}\label{lemma:var_split_and_combine}
Consider $y_s = P^T(v_{T-s+1}-v_{T-s})$ for $s \in [1,T]$. For increments $(v_{T-s+1}-v_{T-s})$ having same covariance matrix $\Sigma$, we have
\begin{align*}
    \sum_{i=1}^d \E[(y_s)_i^2] = \sigma^2
\end{align*}
where $(y_s)_i$ is the $i^{th}$ coordinate of $y$ and $\sigma^2 = Tr(\Sigma)$.
\end{lemma}
\begin{proof}
We take $y_s = P^T (v_{T-s+1}-v_{T-s})$. If all increments $(v_{T-s+1}-v_{T-s})$ have same covariance matrix $\Sigma$, we will have
\begin{align}\label{eqn:same_covariance}
    \E(y_s)_i^2 &= \E[e_i^T P^T (v_{T-s+1}-v_{T-s})(v_{T-s+1}-v_{T-s})^T P e_i] \\
    &= e_i^T P^T \Sigma P e_i\\
    &= \E(y)_i^2.
\end{align}
Therefore,
\begin{align}
\sum_{i=1}^d  \E(y)_i^2 &= \sum_{i=1}^d \E[(v_{T-s+1}-v_{T-s})^T P e_i e_i^T P^T (v_{T-s+1}-v_{T-s})]\\ 
&= \E[(v_{T-s+1}-v_{T-s})^T P \left(\sum_{i=1}^d e_i e_i^T \right) P^T (v_{T-s+1}-v_{T-s})]\\
&= \E[(v_{T-s+1}-v_{T-s})^T P I P^T (v_{T-s+1}-v_{T-s})]\\
&= \E\|(v_{T-s+1}-v_{T-s})\|_2^2\\
&= \sigma^2
\end{align}
\end{proof}

\begin{corollary}\label{corr:NOFI_AI_regret_dim}
For the relaxed assumption that increments only have same variance $\sigma^2$ (and not necessarily the entire covariance matrix), \LAItxt{} has the following constant regret,
\begin{equation*}
\text{Regret}_{\LAImath{}}[1,T] \leq  \frac{\sigma^2}{4} \sum_{i=1}^d \left( \frac{\left(1 + \frac{4}{\lambda_{i}^A}\right)^{\frac{\gamma}{2}} - 1}{\lambda_{i}^A + 2} \right)
\end{equation*}
\end{corollary}

\begin{proof}
For the looser assumption (that the increments only have same variance and not the entire covariance matrix)
\begin{align}
\begin{split}
\E[Cos&t_{\LAImath{}}[1,T]] - \E[\text{Cost}_{\text{\ai{}}}[1,T]]\\
\leq& \frac{1}{2}\sum_{t=1}^{T}  \E\left[(v_{t} - v_{t-1})^T P(D_t-\Tilde{D}_{t}) P^T(v_{t} - v_{t-1})\right]
\end{split}\\
\leq& \frac{1}{2}\sum_{t=1}^{T}  \E\left[(v_{t} - v_{t-1})^T P\left(\sum_{i=1}^d\lambda_i^t - \Tilde{\lambda}_i^t \right) I P^T(v_{t} - v_{t-1})\right]\\
=& \sum_{t=1}^T \left(\sum_{i=1}^d\lambda_i^t - \Tilde{\lambda}_i^t \right) \frac{\sigma^2}{2}\\
=& \frac{\sigma^2}{2} \sum_{i=1}^d \frac{\lambda_i^A}{2}\left( \left(1 + \frac{4}{\lambda_i^A}\right)^{\frac{\gamma}{2}} - 1 \right) \sum_{t=1}^T \left(\lambda_i^T \right)^{2(T-t+1)}\\
\leq& \frac{\sigma^2}{2} \sum_{i=1}^d \frac{\lambda_i^A}{2}\left( \left(1 + \frac{4}{\lambda_i^A}\right)^{\frac{\gamma}{2}} - 1 \right) \sum_{t=1}^\infty \left(\lambda_i^T \right)^{2t}\\
=& \frac{\sigma^2}{4} \sum_{i=1}^d \left( \frac{ \lambda_i^A\left( \left(1 + \frac{4}{\lambda_i^A}\right)^{\frac{\gamma}{2}} - 1 \right)}{\left(\lambda_i^A + 1 \right)^2 -1} \right)\\
&= \frac{\sigma^2}{4} \sum_{i=1}^d \left( \frac{\left(1 + \frac{4}{\lambda_i^A}\right)^{\frac{\gamma}{2}} - 1}{\lambda_i^A + 2} \right)
\end{align}
\end{proof}

\section{Details of Numerical Experiments}\label{numerical_exp_details}
\subsection{Purely-stochastic environment}
In this subsection, we explain how the minimizer sequence $\{v_t\}_t$ is developed for the experiments. First, the minimizer is written as a sum of random variables $\{u_t\}_t$, 
\begin{align}
    v_t = \sum_{s=1}^t u_s
\end{align}
which can be thought of as increments to the minimizers. The different distributions discussed are with respect to the the increments $\{u_t\}_t$. We build $\{u_t\}_t$ such that they are not necessarily independent of each other. The process for any subset of the horizon $\tau \subseteq \{1,\ldots,T\}$ involves the following steps:
\begin{enumerate}
    \item Choosing a common distribution $F_\tau(\cdot)$ and sampling $z_t$
    \begin{align}
        z_t \sim F_\tau \text{ } \forall \text{ } t \in \tau
    \end{align}
    \item Choosing a positive definite covariance matrix $\Sigma \in \R^{10|\tau| \times 10|\tau|}$ and performing its Choleksy decomposition to get $L$, such that $LL^T = \Sigma$ 
    \item Concatenating $\{z_t\}_{t \in \tau}$ into a single vector $z \in R^{10|\tau|}$ and transforming it into
    \begin{align}
        u = L z
    \end{align}
    and splitting it back into $\{u_t\}_{t \in \tau}$
\end{enumerate}
For the light tail case, we want to exhibit the robustness to distribution shift. We, therefore, split the horizon into five subsets $\{\tau_i\}_{i=0}^4$ where $\tau_i = \{ (iT/5) + 1, \ldots, (i+1)T/5 \}$. For each subset we create the increments $u_t$ according to the above mentioned process, each subset pertaining to a different form of distribution.

To show that our results hold for heavy tail distributions, we choose the pareto distribution (Type 2, that is, Lomax \cite{lomax} with parameter $\alpha > 2$ to ensure finite variance) and the log-normal distributions as our candidates. Since these distributions are one sided, we build the $z_t$ vector, in the process above, by first sampling $d=10$ elements from the one dimensional distribution and then multiply a $p=0.5$ Bernoulli random variable to each element, making $z_t$ a zero-mean heavy tail random vector. 

\subsection{Mixed Environment}
Our next set of experiments deal with an environment which is partly stochastic and partly adversarial. We do this by putting adversarial minimizers in an otherwise martingale sequence of minimizers. The level of adversarial infiltration ranges from zero to hundred, with zero being martingale minimizers and hundred being completely adversarial minimizers. We perform this in the following way:
\begin{enumerate}
    \item First, we generate an IID sequence of increments $\{u_t \sim F: t \in \{1,\ldots,T\}\}$ and define a martingale minimizers as $v_t = \sum_{s=1}^t u_s$. Here, the underlying distribution $F$ can be light or heavy tailed.
    \item Next, we decide the percentage $(p)$ of rounds we want to be adversarial and choose that many rounds from $\{1,\ldots,T\}$ uniformly at random. The set of adversarial rounds $\tau$ is fixed for the $N=1000$ sample runs.
    \item During the runs, we replace $\{v_t\}_{t \in \tau}$ with $\{\Tilde{v}_t\}_{t \in \tau}$ which are generated adversarially. Note that $\{\Tilde{v}_t\}_{t \in \tau}$ is fixed for all sample runs of a given adversarial percentage $(p)$.
\end{enumerate}
For light tail, we choose the normal distribution and for heavy tails we take the two-sided pareto and log-normal distributions. The adversarial percentage $(p)$ varies from $0$ to $100$ and we consider a horizon of $T=100$. The scale of total cost for an adversarial environment can be different than that of a stochastic environment. To mitigate this issue, we plot the ratio of the total cost of an online algorithm \alg{} to that of \ai{}, that is, $\frac{\E[\text{Cost}_{\text{\alg{}}}[1,T]]}{\E[\text{Cost}_{\text{\ai{}}}[1,T]]}$. Note that as the environment becomes more adversarial than stochastic, the expectation has no effect.
\end{appendices}

\end{document}